\documentclass[tikz, 11pt, english]{article}

\usepackage[english]{babel}
\usepackage[thms]{packages}
\usepackage{dsfont}
\usepackage{todonotes}

\usepackage{authblk}
\usepackage{environ}

\usepackage[hidelinks]{hyperref}
\usepackage{orcidlink}
\usepackage{cleveref}

\makeatletter
\newsavebox{\measure@tikzpicture}
\NewEnviron{scaletikzpicturetowidth}[1]{%
  \def\tikz@width{#1}%
  \def\tikzscale{1}\begin{lrbox}{\measure@tikzpicture}%
  \BODY
  \end{lrbox}%
  \pgfmathparse{#1/\wd\measure@tikzpicture}%
  \edef\tikzscale{\pgfmathresult}%
  \BODY
}
\makeatother

\usetikzlibrary{shapes, arrows, patterns, patterns.meta}

\usepackage[style=numeric, backend=bibtex, 
maxnames=9, giveninits=true, hyperref=true, backref=false, url=false, isbn=false]{biblatex}

\makeatletter
\newcommand\avsuminner[2]{%
  {\sbox0{$\m@th#1\sum$}%
   \vphantom{\usebox0}%
   \ooalign{%
     \hidewidth
     \smash{\vrule height\dimexpr\ht0+1pt\relax depth\dimexpr\dp0+1pt\relax}%
     \hidewidth\cr
     $\m@th#1\sum$\cr
   }%
  }%
}
\makeatother

\renewcommand{\d}{\mathrm{d}}

\newcommand{\musxhu}[2][]{\todo[color=pink, #1]{Julian: #2}}

\newcommand{\pg}[1]{{\color{cyan} #1}}
\newcommand{\lx}[1]{{\color{darkgreen} #1}}

\DeclareNameAlias{default}{first-last}

\title{A novel approach to hydrodynamics  for\\ long-range generalized exclusion}
\author[1]{Patrícia Gonçalves\orcidlink{0000-0002-8093-8810}}
\author[2]{Julian Kern\orcidlink{0000-0002-8231-0736}}
\author[3]{Lu Xu\orcidlink{0000-0001-9973-6090}}

\affil[1]{CAMGSD, Instituto Superior Técnico, Lisbon, Portugal}
\affil[2]{Weierstraß Institute for applied Analysis and Stochastics, Berlin, Germany}
\affil[3]{Gran Sasso Science Institute, L'Aquila, Italy}

\begin{document}

\maketitle

\begin{abstract}
We consider a class of generalized long-range exclusion processes evolving either on $\mathbb Z$ or on a finite lattice with an open boundary.
The jump rates are given in terms of a general kernel depending on both the departure and destination sites, and it is such that the particle displacement has an infinite expectation, but some tail bounds are satisfied.
We study the superballisitic scaling limit of the particle density and prove that its space-time evolution is concentrated on the set of weak solutions to a non-local transport equation.
Since the stationary states of the dynamics are unknown, we develop a new approach to such a limit relying only on the algebraic structure of the Markovian generator.
\end{abstract}

\section{Introduction}
\bigskip

The field of statistical mechanics is concerned with deriving macroscopic properties of relevant quantities of physical systems from their microscopic description. 
When microscopic models can be described by means of an interacting particle system, the first objective is to derive a law of large numbers, usually referred to as the \emph{hydrodynamic limit} of the system. 
By appropriately rescaling space and time, one hopes to recover a system of partial differential equations (PDEs) that describe the macroscopic space-time evolution of each of the conserved quantities of the system. 
In its simplest form, one may consider $N$ independent symmetric simple random walks on $\mathbb{Z}^d$ for which the usual law of large numbers yields that, on the parabolic scale $(x,t)\mapsto (x/N, t/N^2)$, the space-time evolution of the system's density is given by the heat equation in $\mathbb{R}^d$. 
A natural step, then, is to introduce interaction among particles via, for example,  an \emph{exclusion rule} dictating that jumps to already occupied sites are suppressed. 
In this paper, we will consider a relaxed variant of the exclusion rule that allows up to $\kappa\in\mathbb N$ particles per site, see also \cite{SS94,FGS22}.

The natural object to derive a hydrodynamic limit is the distribution of conserved quantities.
In the case of (generalized) exclusion processes, the only conserved quantity is the total number of particles.
In order to study its distribution, one usually considers the induced empirical distribution of particles.
The associated hydrodynamic limit, then, is a law of large numbers on the space of measures.
In order to obtain a PDE describing the macroscopic dynamics, one needs to ensure that the limiting measure has a density that follows some hydrodynamic equation.

In the case of long-range dynamics, there are multiple phase transitions governing the macroscopic behaviour.
In the setting of symmetric jump kernels, the behaviour of the system is diffusive when the variance of the jump kernel becomes finite, fractional otherwise, see e.g.~\cite{J08}.
For asymmetric jump kernels, the critical point is where the mean of the kernel becomes finite, in which case the system transitions from a fractional scale to a hyperbolic one, see e.g.~
\cite{SS18} in the totally asymmetric case.
An additional difficulty in the setting of long-range systems in the fractional regime is that the non-locality persists at the macroscopic level.
This leads to non-local (or: fractional) hydrodynamic equations for which the solution theory is less developed as in the classical cases.
In the asymmetric case, even the uniqueness of the natural solution concept remains an open problem.
This has the immediate drawback that the classical approach to convergence via tightness and characterization of the limit point is not sufficient to conclude.
Since we are treating more general cases, the same restriction also affects this work.

\subsection{Long-range stochastic dynamics}

Fix a maximal number $\kappa\in\mathbb N$ to be allowed per site.
Denoting the number of particles at a site $x$ and at time $t$ by $\eta_t(x)$, this means that $\eta_t(x)\in\{0,1,\cdots,\kappa\}$ for any $x$ and any $t$. 
To force the dynamics to obey the generalized exclusion rule, we choose the rate for a particle jumping from a site $x$ to a site $y$ to equal $\mathfrak{p}(x,y)\eta(x)(\kappa-\eta(y))$ for some jump kernel $\mathfrak{p}$, cf.~\Cref{ssec:ass_jump_kernel}. 
To fix ideas, one may assume at this moment the jump kernel to take the classical form $\mathfrak{p}(x,y)\sim \vert y-x\vert^{-1-\gamma}$ for some $\gamma \in (0,1)$.
\begin{figure}[htb!]
	\centering
	\begin{scaletikzpicturetowidth}{\textwidth}
	\begin{tikzpicture}[scale=\tikzscale]
\draw[latex-latex] (-7.1,0) -- (7.1,0) ; 
\foreach \x in  {-6,-5,-4,-3,-2,-1,0,1,2,3,4,5,6} 
	\draw[shift={(\x,0)},color=black] (0pt,3pt) -- (0pt,-3pt);
\foreach \x in {-6,-5,-4,-3,-2,-1,0,1,2,3,4,5,6} 
	\pgfmathtruncatemacro{\y}{\x+8}
	\draw[shift={(\x,0)},color=black] (0pt,0pt) -- (0pt,-3pt) node[below] {$\y$};

\node (left) at (-6.8,-12pt) {$\dots$};
\node (right) at (6.8,-12pt) {$\dots$};

\foreach \x in {-6,-4,1,2,5}
	\node[circle,shading=ball,minimum width=0.2cm, ball color=black] (\x) at (\x,0.2) {};
	
	\node (b) at (-3.9,0.3) {};
	\node (a) at (-2.1,0.2) {};
	\node (bb) at (2.1,0.3) {};
	\node (aa) at (4.9, 0.3) {};
	\node (bbb) at (0.9, 0.3) {};
	\node (aaa) at (0, 0.2) {};
	\node (bbbb) at (5.1, 0.3) {};
	\node (aaaa) at (5.9, 0.2) {};
	\node (c) at (-5.9, 0.3) {};
	\node (d) at (-1.1, 0.2) {};

    \draw[->] (b.north) to [out=60, in=120]  (a.north);
    \node at (-3,1.1) {$\mathfrak{p}(4,6)$};
    \draw[->, color=red] (bb.north) to [out=60, in=120]  (aa.north);
    \node[color=red] at (3.5,1.15) {$/$};
    \draw[->] (bbb.north) to [out=120, in=60]  (aaa.north);
    \node at (0.5,0.85) {$\mathfrak{p}(9,8)$};
    \draw[->] (bbbb.north) to [out=60, in=120]  (aaaa.north);
    \node at (5.5,0.9) {$\mathfrak{p}(13,14)$};
    \draw[->] (c.north) to [out=60, in=120]  (d.north);
    \node at (-3.5,2) {$\mathfrak{p}(2,7)$};
	\end{tikzpicture}
	\end{scaletikzpicturetowidth}
	\caption{Illustration of the bulk dynamics of the long-range exclusion process ($\kappa = 1$) on $\mathbb Z$.}\label{fig:bulk_dyn}
\end{figure}
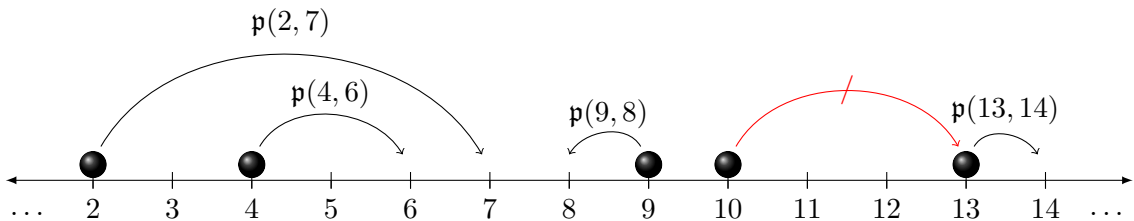

Additionally to the particle system on $\mathbb Z$, we also want to allow for dynamics on a bounded region in contact with reservoirs.
Following \cite{BGJ21,BGOj19,BCGS22}, we will consider the above dynamics on the \emph{bulk} $\{1,2,\dots, N-1\}$ and replace all other sites of $\mathbb Z$ by stochastic reservoirs, see also \cite{S21}.
These interact with the bulk through the same type of dynamics, the reservoirs acting as if they would contain some density of particles.
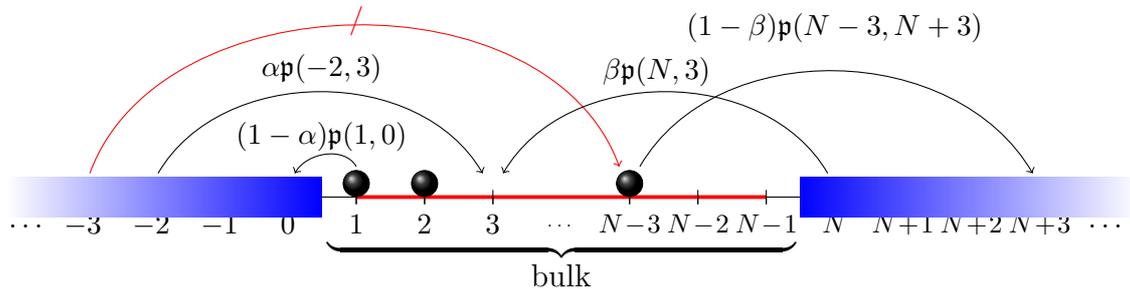
\begin{figure}[ht]
	\centering
	\begin{scaletikzpicturetowidth}{\textwidth}
	\begin{tikzpicture}[scale=\tikzscale]
\draw[latex-latex] (-4.1,0) -- (12.5,0) ; 
\draw[line width=0.5mm, color=red] (1,0) -- (7,0);
\foreach \x in  {-3,-2,-1,0,1,2,3,5,6} 
	\draw[shift={(\x,0)},color=black] (0pt,3pt) -- (0pt,-3pt);
\node (bulkmid) at (4, -12pt) {\tiny{$\dots$}};
\draw[shift={(7,0)},color=black] (0pt,3pt) -- (0pt,-3pt);
\draw[shift={(8,0)},color=black] (0pt,3pt) -- (0pt,-3pt);
\draw[shift={(5,0)},color=black] (0pt,0pt) -- (0pt,-3pt) node[below] {\small{$N\!-\!3$}};
\draw[shift={(6,0)},color=black] (0pt,0pt) -- (0pt,-3pt) node[below] {\small{$N\!-\!2$}};
\draw[shift={(7,0)},color=black] (0pt,0pt) -- (0pt,-3pt) node[below] {\small{$N\!-\!1$}};
\draw[shift={(8,0)},color=black] (0pt,0pt) -- (0pt,-3pt) node[below] {\small{$N$}};
\foreach \x in  {1,2,3} 
	\draw[shift={(8+\x,0)},color=black] (0pt,3pt) -- (0pt,-3pt);
\foreach \x in {-3,-2,-1,0,1,2,3} 
	\draw[shift={(\x,0)},color=black] (0pt,0pt) -- (0pt,-3pt) node[below] {$\x$};
\foreach \x in {1,2,3}
	\draw[shift={(\x+8,0)}, color=black] (0pt,0pt) -- (0pt,-3pt) node[below] {\small{$N\!+\!\x$}};

\node (left) at (-3.8,-12pt) {$\dots$};
\node (right) at (12,-12pt) {$\dots$};

\node (bulktext) at (4,-1) {$\underbrace{\qquad\qquad\qquad\qquad\qquad\qquad\qquad\qquad}_{\text{\large{bulk}}}$};

\foreach \x in {1,2,5}
	\node[circle,shading=ball,minimum width=0.2cm, ball color=black] (\x) at (\x,0.2) {};
	
	\node (b) at (1,0.3) {};
	\node (bb) at (2,0.3) {};
	\node (bbb) at (5,0.3) {};
	\node (nb) at (3, 0.2) {};
	\node (nbb) at (6,0.2) {};
	\node (nbbb) at (7,0.2) {};
	
	\node (l) at (-2.9, 0.2) {};
	\node (ll) at (-1.9, 0.2) {};
	\node (lll) at (-0.9, 0.2) {};
	\node (llll) at (0.1, 0.2) {};
	
	\node (r) at (7.9, 0.2) {};
	\node (rr) at (8.9,0.2) {};
	\node (rrr) at (9.9, 0.2) {};
	\node (rrrr) at (10.9, 0.2) {};

    \draw[->] (ll.north) to [out=60, in=120]  (nb.north west);
    \node at (0.5,1.9) {$\alpha\mathfrak{p}(-2, 3)$};
    \draw[->, color=red] (l.north) to [out=70, in=110] (bbb.north west);
    \node[color=red] at (1,2.6) {$/$};
	\draw[->] (b.north) to [out=120, in=60]  (llll.north);
	\node at (0.5, 0.9) {$(1-\alpha)\mathfrak{p}(1,0)$};
	
	\draw[->] (r.north) to [out=120, in=60] (nb.north east);
	\node at (5.4,1.85) {$\beta\mathfrak{p}(N, 3)$};
	
	\draw[->] (bbb.north east) to [out=60, in=120] (rrrr.north);
	\node at (8,2.5) {$(1-\beta)\mathfrak{p}(N-3, N+3)$};
	
	\shade[right color=blue, left color=white, opacity=0.4]
      (-4.1, -0.3) rectangle (0.5,0.3);
	\shade[left color=blue, right color=white, opacity=0.4]
      (7.5,-0.3) rectangle (12.5,0.3);
	\end{tikzpicture}
	\end{scaletikzpicturetowidth}
	\caption{Illustration of the reservoir dynamics for the long-range exclusion process ($\kappa = 1$) with density $\alpha\in [0,1]$ to the left and $\beta\in [0,1]$ to the right.}\label{fig:res_dyn}
\end{figure}

For dynamics with only finite-range jumps, the seminal paper \cite{R91} investigated the case of translation-invariant jump distribution $\mathfrak p$ with drift on $\mathbb Z^d$, and showed the convergence to Burgers' equation under the hyperbolic scale $(x,t)\mapsto (x/N, t/N)$.
The result has been extended by \cite{B12} to systems evolving on a subset of $\mathbb{Z}^d$ and in contact with reservoirs at the boundary, see also \cite{X24,X22}.
For dynamics with long-range jumps,
\cite{J08} first studied the symmetric exclusion processes (when $\kappa=1$) on $\mathbb{Z}^d$ driven by a translation-invariant jump kernel of the form \lx{$\mathfrak{p}(x,y) = \mathfrak p(|y-x|) \sim \vert y-x\vert^{-d-\gamma}$}, $\gamma\in (0,2)$, and showed the convergence to a fractional heat equation under a superdiffusive scale.
The corresponding boundary-driven dynamics have been treated in a series of papers \cite{BGJ21,BGOj19,BCGS22} in one dimension.
In the asymmetric case, the situation is analogous in the sense that one obtains for $\gamma\in(0,1)$ a fractional Burgers' equation in the limit under the superballistic 
scale $(x,t)\mapsto(Nx,N^\gamma t)$.
As soon as $\gamma \geq 1$, the mean jump size becomes finite and the system behaves as the entropy solution to the classical Burgers' equation under the hyperbolic scale.
These results are proved in \cite{SS18} for translation-invariant $\mathfrak p(x,y)=(y-x)^{-d-\gamma}\mathds1_{x<y}$.
As mentioned before, the uniqueness of the notion of a weak solution to the fractional Burgers' equation proposed therein is still an open problem.

\subsection{Our contributions}

The analysis of exclusion-type processes usually relies on the entropy method pioneered by \cite{GPV} to prove the replacement lemmas necessary to pass from the microscopic model to the macroscopic limit.
For this approach to work, one has to have a good control over so-called reference measures that are usually stationary measures or inspired by them.
It has been shown to work for the symmetric long-range case \cite{J08} as well as in the more general setting in \cite{SS18} in the case of totally asymmetric jump kernels.
In both cases, the jump kernel was taken to be of the special form $\vert y-x\vert^{-1-\gamma}$ in the one-dimensional model.
Since the computation for the stationary measures usually require translation-invariant jump kernels, it is unclear whether this approach can be generalized to other kernels.

In this work, we propose a new approach that circumvents all these difficulties.
Instead of relying on the entropy method, we are able to prove the necessary replacement lemmas directly by using the algebraic structure of the generator for the generalized long-range exclusion processes to obtain the hydrodynamic behaviour for a wide range of transition kernels.
The crucial part of the proof is the replacement of microscopic expressions of the form
\[
\dfrac{1}{N^2}\sum_{x\neq y} H\left(\dfrac{x}{N},\dfrac{y}{N}\right)\eta(x)\big(\kappa - \eta(y)\big)\qquad\text{ by }\qquad \dfrac{1}{N^2}\sum_{x\neq y} H\left(\dfrac{x}{N},\dfrac{y}{N}\right)\eta(x)^{\epsilon N}\left(\kappa - \eta(y)^{\epsilon N}\right),
\]
where $\eta(x)^{\epsilon N}$ denotes the average occupation number over a box of size $2\epsilon N + 1$.
Using the fact that the map $(\mathsf{s},\mathsf{s}')\mapsto \mathsf{s}(\kappa - \mathsf{s}')$ is affine in both $\mathsf{s}$ and $\mathsf{s}'$, the necessary replacement lemma can be translated into a statement on the analytic properties of $H$.

In practice, $H$ is a combination of the jump kernel $\mathfrak{p}$ and some expression involving smooth test functions.
That allows us to formulate rather weak assumptions on the jump kernel in order for the statement to hold.
More precisely, one needs the jump kernel to be locally Lipschitz and ``spread out enough", the latter being an assumption on the process to be long-range in the sense that the above expression does not concentrate close to the diagonal $\{x = y\}$.
A final assumption on the jump kernel is to be homogeneous, so that it behaves well w.r.t.~the time scaling.
If one is willing to allow for microscopic generators that depend on the scaling parameter, the class of admissible jump kernels can be extended to a large class of locally Lipschitz continuous functions.

Compared to the relative entropy method used to prove the hydrodynamic limit\footnote{The proof in \cite{SS18} has some restriction on the initial condition. In the setting of this work, corresponding to the case $\gamma \in (0,1)$, however, one observes that their proof goes through without these restrictions.} for the classical jump kernel $\vert y-x\vert^{-1-\gamma}$, $\gamma\in (0,1)$ in \cite{SS18},
our approach allows to treat much more general jump kernels that are not necessarily translation-invariant, with a much simpler proof.\\

Additionally to the generalized exclusion process on $\mathbb Z$, we also consider the associated boundary-driven dynamics.
In the case $\kappa = 1$ with a symmetric long-range jump kernel of the form $\vert y-x\vert^{-1-\gamma}$, the hydrodynamic limits have been investigated in \cite{BGJ21,BGOj19,BCGS22}.
In the asymmetric case, only the totally asymmetric exclusion process has been treated, see \cite{K24}.
Although only the classical jump kernel $\vert y-x\vert^{-1-\gamma}$ has been treated explicitly, the method therein can be applied to other totally asymmetric systems.

In this paper, we obtain a hydrodynamic equation under similar assumption on the jump kernel as in the case without boundaries.
In order to obtain useful information on the boundary behaviour, however, we have to restrict ourselves to a smaller class of jump kernels $\mathfrak{p}$: 
when the symmetric part $\mathfrak s(x,y)$ and the anti-symmetric part $\mathfrak a(x,y)$ of the jump kernel are bounded by $c|x-y|^{-1-\gamma}$ with some positive constant $c$ from both above and below, we are able to obtain Dirichlet-type boundary conditions.

\subsection{Possible extensions}

\paragraph{Other geometries.}
In this work, we restrict ourselves to processes on $\mathbb Z$ for readability, but the same strategy can also be used to treat generalized exclusions in higher dimensions and also on different discretizations, e.g.~the triangular or hexagonal lattices.
In the case of boundary-driven systems, additional care has to be taken in the approximation of the considered bounded domain, see e.g.~\cite{B12,SPS24}.

\paragraph{Multi-component systems.}

Another possible generalization is towards multi-type models.
Instead of considering one type of particles, one considers two types of particles, e.g.~positive and negative particles.
They are then couples in the very natural way of imposing the generaliezed inclusion rule onto the total number of particles at a given site.
Then, the above jump rate becomes $\eta^\pm(x)\big(\kappa - \eta^+(y) - \eta^-(y)\big)$ for positive and negative particles respectively.
Since this expressions remains affine in the occupation variables, it is to be expected that our approach can be adapted to this context.

\paragraph{Inclusion processes.}
The main novelty of our approach is that we make use of the bilinearity of the generator expressions.
In the one-dimensional case, this restricts its application, up to a time rescaling, to three types of processes: independent random walks, generalized exclusion processes and inclusion processes.
The latter replace the expression $\eta(x)\big(\kappa - \eta(y)\big)$ by $\eta(x)\big(\alpha + \eta(y)\big)$, encouraging the particles to jump onto crowded sites.

Unfortunately, inclusion processes are much harder to handle.
Since particles jump preferentially onto crowded sites, the occupation numbers are unbounded \emph{a priori}.
This leads already to difficulties in rigorously defining the process.
In the symmetric case, duality arguments can be employed to define the process on $\mathbb Z$ for a class of initial conditions for which all mixed moments of occupation numbers are uniformly bounded, see \cite{GRV10,KR16} for details.
In the asymmetric case, it is unclear if and under which conditions existence can be ensured.
For the boundary-driven dynamics, the problem is more amenable as the bulk dynamics conserve the total number of particles, so that the total number of particles remains finite.
If one allows for boundary interactions, however, we are not aware of an argument that ensures existence in the asymmetric case; in the symmetric case, see \cite{FGS22}.

In the proof of the main replacement lemma, we have taken care to weaken the assumption on the particle configuration.
Instead of assuming the occupation number to be uniformly bounded, we only assume a bound that is uniform in space and $L^2$ in time.
In the symmetric case, this bound is satisfied by the inclusion process whenever started from one of the initial conditions described above.
It is our hope that a similar estimate can be shown to hold in some asymmetric regimes, which then would imply the hydrodynamic limit through the arguments exposed in this paper.

\paragraph{Fast boundary dynamics.}
In this work, we do not consider fast boundary dynamics, i.e.~boundary dynamics that act on a shorter time scale than the bulk dynamics.
The corresponding hydrodynamic limit does not really depend on the bulk dynamics and does not require the same replacement lemmas as the boundary term is linear in the occupation variable, see \cite{BGJ21} for details.

\subsection{Remarks and open questions}
Below we describe some problems related to the results of this article which are left for future work.
\paragraph{Uniqueness and regularity of the hydrodynamic equation.}
As in the totally asymmetric setting \cite{SS18}, no appropriate uniqueness result is available for the hydrodynamic equation we derive.
In the boundary-driven case, we derive Dirichlet-type boundary conditions for a range of singular jump kernels, but it remains unclear whether these are enough to characterize the solution uniquely.
Investigating the uniqueness of the proposed weak solution concept, as well as identifying additional properties sufficient for uniqueness to hold, remains an open problem.

Simulations seem to indicate that the solution, even in the asymmetric case, is continuous or more regular.
Unfortunately, the arguments used in \cite{BGJ21,BCGS22} to show an appropriate energy estimate and, with it, Sobolev regularity, cannot be generalized to asymmetric jump kernels. 
It therefore remains an open problem to show regularity from the microscopic level or to provide a new argument for some energy estimate to hold.

\paragraph{Fluctuations.}
Once the hydrodynamic limit is derived, it is natural to investigate the fluctuations of the microscopic system around the typical profile. To that end, two settings are considered: the simpler approach is to start the system from its stationary measure if that measure is known explicitly; the more challenging approach is to start the system from a general measure, typically satisfying the same conditions as those required for hydrodynamic limits. In this case, the limit density fluctuation field is expected to be a solution to a stochastic PDE that includes an additional noise term compared to the hydrodynamic equation, usually in the form of a space-time white noise. 

There are many results in the literature about equilibrium fluctuations but in the setting of long-range dynamics, few results are known. The non-equilibrium scenario is even much less understood.
In the case of the translation invariant jump kernel $\mathfrak{p}(x,y) = c^\pm \vert y -x\vert^{-1-\gamma}$ with $c^\pm$ depending on the sign of $y-x$, the equilibrium fluctuations have been derived in \cite{GM17}, see also \cite{S21}.
But even in the symmetric setting, out-of-equilibrium fluctuations remain an open question for long-range systems, with the primary difficulty being the control of correlation estimates among occupation variables.

\paragraph{Tagged particles.}

Another question of interest relates to the Law of Large Numbers or Central Limit Theorem results for observables of the model other than the particle density. 
An intriguing question is to derive these results for the integrated current of particles in the system and a tagged particle. 
In the nearest-neighbour case, two settings are distinguished: first, for certain symmetric rates, the limits are fractional Brownian motions; however, in the case of certain asymmetric rates, the limits become Brownian motions; see e.g.~\cite{GM08, SV23} and references therein.

In the setting of long-range interactions, much less is known.
To our knowledge, no result has been obtained in the symmetric setting.
The recent work \cite{Z24} obtains a Lévy-description of the tagged particle for the totally asymmetric kernel, corresponding in one dimension to $(y-x)^{-1-\gamma}\mathds{1}_{y> x}$, with $\gamma \in (0, 3/2)$.
The method therein relies crucially on the translation-invariance of the jump kernel to obtain Markovian properties for the induced tagged particle.
It would be interesting to obtain analogous results in the context of general jump kernels that should lead to Lévy-type processes for the dynamics of a tagged particle.

\subsection{Outline}

The article is structured in three additional sections. 
In \Cref{sec:notation}, we introduce important notation that we will use throughout. 
\Cref{sec:gen_exclusion} is dedicated to the generalized exclusion evolving on $\mathbb Z$. 
Therein, we introduce the model and its main properties, we state the hydrodynamic limit and we present its proof. 
In \Cref{sec:gen_exclusion_boundary}, we present the necessary modifications to treat the generalized exclusion in contact with (slow) reservoirs.

\vspace{1cm}

\paragraph{Acknowledgments.}

P. Gonçalves thanks  FCT/Portugal for financial support through the projects UIDB/04459/2020, UIDP/04459/2020 and the project SAUL funded by FCT-ERC.
J. Kern has been funded by the Deutsche Forschungsgemeinschaft (DFG, German Research Foundation) under Germany's Excellence Strategy – The Berlin Mathematics Research Center MATH+ (EXC-2046/1, project ID: 390685689) and through grant CRC 1114 ``Scaling Cascades in Complex Systems", Project Number 235221301, Project C02 ``Interface dynamics: Bridging stochastic and hydrodynamic descriptions".
L. Xu thanks the financial support from the project Dipartimento di Eccellenza of MUR, Italy.

\section{Notation}\label{sec:notation}

Throughout this paper, let $\mathbb N=\{1,2,\ldots\}$ be the set of positive natural numbers, and denote $\mathbb N_0:=\mathbb N \cup \{0\}$.
We write $\mathbb R_+:=[0,+\infty)$ for the time line and $I=\mathbb R$ or $[0,1]$ for the spatial domain.
Let $C_c(I)$ be the set of compactly supported continuous functions $g:I\to\mathbb R$, and $C_c(\mathbb R_+ \times I)$ be the set of continuous functions $G:\mathbb R_+ \times I \to \mathbb R$ that are compactly supported in both space and time.
Denote by $C_c^\infty(I)$ the set of all smooth functions in $C_c(I)$, endowed with the topology induced by uniform convergence of all derivatives. 
Denote by $C_c^{1,\infty}(\mathbb R_+ \times I)$ the set of all functions in $C_c(\mathbb R \times I)$ that are continuously differentiable in time and smooth in space, endowed with the similar topology.
Given a function $G$ on $\mathbb R_+\times I$, we write $G_t = G(t,\cdot)$ for each $t\in\mathbb R_+$.
For a Lebesgue-measurable set $A \subseteq \mathbb R^d$, 
let $L^\infty(A)$ be the set of essentially bounded functions on $A$.
Also denote by $L^1(A)$ (resp.~$L_\mathrm{loc}^1(A)$) the set of integrable (resp. locally integrable) functions on $A$ w.r.t.~the Lebesgue measure.

Let $\mathcal{M}(I)$ be the set of positive Radon measures on $I$.
The integral of $g\in C_c(I)$ with respect to $\pi\in\mathcal{M}(I)$ is denoted by $\langle \pi,g \rangle$.
$\mathcal M(I)$ is a Polish space when endowed with the vague topology, i.e.~$\pi_n$ converges to $\pi$ if and only if $\lim_{n\to\infty} \langle \pi_n,g \rangle=\langle \pi,g \rangle$ for all $g \in C_c(I)$.
Furthermore, fix some $\kappa>0$ and let $\mathcal M_\kappa(I)$ be the set of all measures $\pi\in\mathcal M(I)$ that are absolutely continuous w.r.t.~the Lebesgue measure, and the Radon-Nikodym density is below $\kappa$ almost everywhere:
\begin{equation}\label{eq:meas-space}
\mathcal{M}_\kappa(I) := \{\pi\in\mathcal{M}\;:\; \pi(\mathrm{d}u) = \rho(u)\mathrm{d}u \text{ and } 0\leq \rho\leq \kappa\text{ a.e.}\}.
\end{equation}
We drop the reference to $I$ in $\mathcal M(I)$, $\mathcal M_\kappa(I)$ if it is clear from the context.

%
For a Polish space $E$, we denote the space of càdlàg paths $\mathbb R_+\to E$ endowed with the usual Skorokhod topology by $\mathbb D(\mathbb{R}_+;E)$.
We write $C(\mathbb{R}_+; E)\subseteq \mathbb D(\mathbb{R}_+;E)$ for the subspace of continuous functions $\mathbb R_+\to E$, endowed with the topology induced by local uniform convergence.

Given two functions $f$ and $f'$ on $\mathbb R^d$, the notion $f \lesssim f'$ means that $f \le Cf'$ with some universal constant $C$.
If the constant $C$ depends on some parameter $\alpha$, we write $f \lesssim_\alpha f'$ to emphasise the dependence.
These notions extend to the case of two sequences $\{a_n\}$ and $\{a'_n\}$ in a natural way.

\paragraph{Particle configurations.}
The microscopic dynamics considered in this paper are defined either on the one-dimensional inifinite lattice $\mathbb Z$, or on the finite one $\Lambda_N:=\{1,\ldots,N-1\}$ for an integer $N\ge2$.
Fix $\kappa \in \mathbb N$ as the maximal occupation number. 
Then, the configuration space is $\Omega:=\{0,1,\ldots,\kappa\}^\mathbb Z$ or $\Omega_N:=\{0,1,\ldots,\kappa\}^{\Lambda_N}$.
For a configuration $\eta\in\Omega$ (resp.~$\eta\in\Omega_N$) and a site $x\in\mathbb Z$ (resp.~$x\in\Lambda_N$), we call $\eta(x)$ the \emph{occupation variable} at $x$.

For $x$, $y\in\mathbb Z$ (resp.~$x$, $y\in\Lambda_N$), we define the map $\eta \mapsto \eta^{x,y}$ for $\eta\in\Omega$ (resp.~$\eta\in\Omega_N$) as follows: if $\eta(x) = 0$ or $\eta(y) = \kappa$, set $\eta^{x,y} := \eta$; otherwise,
\[
\eta^{x,y}(z) := \begin{cases}
    \eta(z) &\text{ if } z\not\in\{x,y\}\\
    \eta(x) - 1 & \text{ if } z= x\\
    \eta(y) + 1 & \text{ if } z= y
\end{cases}.
\]
The configuration $\eta^{x,y}$ corresponds to the action that moves one particle from $x$ to $y$.
Furthermore, define the actions $\eta\mapsto \eta^{x,+}$ and $\eta^{x,-}$, corresponding to adding or removing a particle: $\eta^{x,+} := \eta$ if $\eta(x) = \kappa$, $\eta^{x,-} := \eta$ if $\eta(x) = 0$, and otherwise
\[
\eta^{x,\pm}(z) := \begin{cases}
    \eta(z) & \text{ if }z\neq x\\
    \eta(x) \pm 1 & \text{ if } z = x
\end{cases}.
\]

For $\ell\in\mathbb N$, let $\Gamma_\ell := \{z \in\mathbb Z\;:\; \vert z\vert\leq \ell\}$ denote the box around the origin of the size $\vert \Gamma_\ell\vert=2\ell + 1$.
Define the block average in $x+\Gamma_\ell$ through
\[
\eta(x)^\ell := \dfrac{1}{\vert \Gamma_\ell\vert}\sum_{z\in\Gamma_\ell} \eta(x+z).
\]

\paragraph{Naming conventions.}
We denote the time variables by $T,t,r,s\in\mathbb R_+$, the microscopic space variables by $x,y,z\in\mathbb Z$ or $\Lambda_N$, the macroscopic space variables by $u,v, w \in I$, the test functions depending only on the space variable by $g$, the test functions depending both on the time and space variables by $G$, the configurations of the state-space by $\eta$, the measures in $\mathcal{M}$ by $\pi$, and the macroscopic profile by $\rho$.

\paragraph{Abuse of notation.}
Throughout this paper, the lower (resp.~upper) summation limits are meant to be the largest integer below (resp.~smallest integer above) whenever we use a real number.
The first convention is also assumed in the context of block averages.
Moreover, we adopt the following type of contraction to improve readability:
\[
\limsup_{\epsilon\downarrow 0, N\uparrow \infty} := \limsup_{\epsilon\downarrow 0} \limsup_{N\uparrow\infty}.
\]
Namely, the subscripts in the limit shall be read from right to left.

\section{Generalized exclusion process}\label{sec:gen_exclusion}
\label{sec:gen-ex}

For the entirety of this section, we consider the spatial domain $I = \mathbb R$. 
The corresponding microscopic lattice is $\mathbb Z$ and the configuration space is $\Omega$.

{Let $\mathfrak p=\mathfrak{p}(x,y)$ be a \emph{nice} non-negative function defined on $\{(x,y)\in\mathbb Z^2;x\not=y\}$ which will be specified later.}
The generator of a \emph{$\kappa$-generalized exclusion} associated with the jump kernel $\mathfrak{p}$ is defined as
\begin{equation}
\mathcal Lf(\eta) := \sum_{(x,y)\in\mathbb Z^2} \mathfrak{p}(x,y) \eta(x)\big(\kappa - \eta(y)\big) \Big( f(\eta^{x,y}) - f(\eta) \Big),
\end{equation}
where $f:\Omega\to\mathbb R$ is any local\footnote{A function $f$ is said to be \textit{local} if $f$ depends on $\{\eta(x);x \in A\}$ for a finite set $A\subseteq\mathbb Z$.} function.
{In the definition above, we adopt the convention that $\mathfrak{p}(x,x)\equiv0$ for $x\in\mathbb Z$.}

\subsection{Assumptions on the jump kernel}\label{ssec:ass_jump_kernel}

Assume that the jump kernel can be extended to a function on
\[
\mathbb R_0^2:=\{(u,v)\in\mathbb R^2\;:\; u\not=v\},
\]
which is still denoted by $\mathfrak{p}=\mathfrak{p}(u,v)$, such that
\begin{itemize}
\item[(A1)] $\mathfrak{p}$ is Lipschitz continuous on any compact subset of $\mathbb R_0^2$;
\item[(A2)] $\mathfrak{p}$ is $-(1+\gamma)$-homogeneous with some constant $\gamma>0$:
\begin{equation}\label{eq:gamma}
    \mathfrak{p}(\lambda u, \lambda v) = \lambda^{-(1+\gamma)}\mathfrak{p}(u,v),
\end{equation}
for all $(u,v)\in\mathbb R_0^2$ and $\lambda>0$;
\item[(A3)] let $\mathfrak{s}=\mathfrak{s}(u,v)$ be the symmetric part of $\mathfrak{p}$ defined by
\begin{equation}\label{eq:sym_p}
\mathfrak{s}(u,v) = \mathfrak{s}(v,u) := \frac{\mathfrak{p}(u,v)+\mathfrak{p}(v,u)}2, \qquad (u,v)\in\mathbb R_0^2,
\end{equation}
then for any compact set $K\subset\mathbb R$,
\begin{equation}\label{eq:integrability}
\int_K \int_\mathbb R \mathfrak{s}(u,v) \min\big\{|u-v|,1\big\}\,\d v\,\d u < +\infty.
\end{equation}
\end{itemize}
Notice that a continuous extension that satisfies \eqref{eq:gamma} is unique if it exists, since the value of $\mathfrak{p}(q,q')$ is fixed for any rational numbers $q\not=q'$.
We also assume that
\begin{equation}\label{eq:p-sum}
\sup_{x\in\mathbb Z} \left\{ \sum_{y\not=x} \mathfrak{s}(x,y) \right\} < +\infty,
\end{equation}
and for any $R>0$ one has the tail estimates
\begin{align}
\label{eq:tail-1}
&\limsup_{\epsilon\downarrow 0, N\uparrow\infty} \frac{N^\gamma}{N^2}\sum_{|x|<RN} \sum_{|y-x| \le \epsilon N} \mathfrak{s}(x,y) |x-y| =0,\\
\label{eq:tail-2}
&\limsup_{m\uparrow\infty, N\uparrow\infty} \frac{N^\gamma}N \sum_{|x|<RN} \sum_{|y| \ge mN} \mathfrak{s}(x,y)= 0.
\end{align}

\begin{example}[Translation-invariant case]
\label{exa:translation-invariant}
Suppose that $\mathfrak{p}(x,y)$ satisfies the conditions above and is furthermore translation-invariant: $\mathfrak{p}(\cdot,\cdot)=\mathfrak{p}(\cdot+z,\cdot+z)$ for all $z\in\mathbb Z$.
Then, $\mathfrak{p}(x,y)=\mathfrak{p}_*(y-x)$ with\
\[
\mathfrak{p}_*(z) := \big(c\mathds1_{z<0}+c'\mathds1_{z>0}\big)|z|^{-1-\gamma}, \qquad \gamma\in(0,1),
\]
for constants $c$, $c'\ge0$.
In particular, $\mathfrak{p}_\mathrm{sy}(x,y):=\mathds1_{x\not=y}|x-y|^{-1-\gamma}$ generates the symmetric $\kappa$-exclusion (cf. \cite{J08}), while $\mathfrak{p}_\mathrm{asy}(x,y):=\mathds1_{x<y}|x-y|^{-1-\gamma}$ corresponds to the totally asymmetric $\kappa$-exclusion (cf. \cite{SS18}).
\end{example}

\begin{example}
More general examples of $p$ are given by
\[
\mathfrak{p}(x,y)=g(\theta)|x-y|^{-1-\gamma}, \qquad \gamma\in(0,1),
\]
where $\theta \in \mathbb T_{2\pi} \simeq [0,2\pi)$ is the angle of $(x,y)\in\mathbb R^2_0$ in the polar coordinates and $g$ is a bounded function that is $C^1$ on both $(-\tfrac\pi4,\tfrac\pi4)$ and $(\tfrac\pi4,\tfrac{5\pi}4)$.
We shall underline that $\gamma$ is not necessarily below $1$, for instance
\[
\mathfrak{p}(x,y)=\frac{\mathds1_{x<y}}{(|x|+|y|)^{\gamma/2}|x-y|^{1+\gamma/2}}, \qquad 
\gamma\in(0,2),
\]
and the corresponding symmetric version also satisfies our conditions.
\end{example}

\begin{remark}[A sufficient condition]
\label{rem:suff-cond}
Under (A2), a sufficient condition for \eqref{eq:integrability}--\eqref{eq:tail-2} that can be easily verified is $\gamma\in(0,1)$ and
\[
M := \sup \big\{\mathfrak{p}(u,v)\;:\;|u-v|=1\big\} < +\infty.
\]
Indeed, for any $(u,v)\in\mathbb R_0^2$, from \eqref{eq:gamma} we have
\[
\big|\mathfrak{p}(u,v)\big| = \left| \mathfrak p \left( \frac u{|u-v|},\frac v{|u-v|} \right) \right| |u-v|^{-1-\gamma} \le M|u-v|^{-1-\gamma}.
\]
Then all the conditions are satisfied since $\gamma\in(0,1)$.
\end{remark}
%
%

\subsection{Hydrodynamic limit}

From \eqref{eq:p-sum}, the closure of $\mathcal L$ generates a Feller semigroup on the space of bounded functions on $\Omega$, see \cite[pp. 27, Theorem I.3.9]{Liggett}.
Let $\mu_N$ be some probability measure on $\Omega$, and ${\eta^N=\{\eta_t^N\}_{t\ge0}}$ be the accelerated Markov process generated by $N^\gamma\mathcal L_N$ with initial distribution $\mu_N$.
Denote by $\mathbb{P}_{\mu_N}$ its law on $\mathbb D(\mathbb R_+;\Omega)$.
In the following contents, we will often omit the superscript $N$ in $\eta_t^N$ when there is no confusion.

Given $N\in\mathbb N$, define the empirical distribution of a configuration $\eta\in\Omega$ by
\begin{equation}\label{em_meas}
\pi^N(\eta) := \dfrac{1}{N}\sum_{x\in\mathbb Z} \eta(x)\delta_{x/N},
\end{equation}
where $\delta_u$ denotes the Dirac measure at $u\in\mathbb R$.
Recall that $\mathcal M=\mathcal M(\mathbb R)$ is the space of positive Radon measures on $\mathbb R$ and note that $\pi^N(\eta)\in\mathcal M$.

Assume that the sequence of initial distributions $\mu_N$ is associated to some measurable profile $\rho_\circ:\mathbb R\to[0,\kappa]$, in the sense that a law of large numbers for the empirical measure holds in probability at $t=0$: for all $g \in C_c(\mathbb R)$ and $\delta>0$,
\begin{equation}\label{eq:initial}
\lim_{N\uparrow\infty} \mu_N\Big( \left\{ \eta\in\Omega\;:\; \left\vert \big\langle \pi^N(\eta),g \big\rangle - \langle \rho_\circ, g \rangle \right\vert > \delta \right\} \Big) = 0.
\end{equation}

We hope to prove that the above convergence holds at any macroscopic time $t\in\mathbb R_+$, with the empirical profile evolving w.r.t.~the following \emph{hydrodynamic equation}:
\begin{equation}\label{eq:hydrodynamics}
\partial_t\rho_t(u)=\mathbb L\rho_t(u), \qquad \rho_0(u)=\rho_\circ(u), \qquad (t,u)\in\mathbb R_+\times\mathbb R,
\end{equation}
where $\rho_t:=\rho(t,\cdot)$ and $\mathbb L$ is a nonlocal operator given, heuristically, by
\begin{equation}
\mathbb L\rho(u) = \int_\mathbb R \Big[\mathfrak{p}(v,u)\rho(v)\big(\kappa-\rho(u)\big) - \mathfrak{p}(u,v)\rho(u)\big(\kappa-\rho(v)\big)\Big] \,\d v.
\end{equation}
Rigorously speaking, 
for each $\rho \in L^\infty(\mathbb R)$, $\mathbb L\rho$ is a bounded linear functional on $C_c^\infty(\mathbb R)$, such that for all $g\in C_c^\infty(\mathbb R)$, $(\mathbb L\rho)(g)=\langle \mathbb L\rho, g \rangle$ is given by
\begin{equation}\label{eq:operator_L}
\langle \mathbb L\rho, g \rangle := \iint_{\mathbb R^2} \mathfrak{p}(u,v)\big[g(v)-g(u)\big]\rho(u)\big(\kappa-\rho(v)\big) \,\d v\,\d u.
\end{equation}

\begin{definition}\label{def:weak-sol}
A \emph{weak solution} to \eqref{eq:hydrodynamics} is a measurable function $\rho:\mathbb R_+\times\mathbb R\to[0,\kappa]$, such that for all test functions $G \in C_c^{1,\infty}(\mathbb R_+\times \mathbb R)$,
\begin{equation}\label{eq:weak-sol}
\langle \rho_\circ,G_0 \rangle + \int_0^\infty \langle \rho_t, \partial_tG_t \rangle \,\d t + \int_0^\infty \langle \mathbb L\rho_t,G_t \rangle\,\d t = 0.
\end{equation}
\end{definition}
\begin{remark}
If $\mathfrak{p}(x,y)=\mathfrak p_*(|x-y|)$ is translation-invariant as well as symmetric, then $\mathbb L$ becomes the pseudo-Laplacian:
\begin{align*}
\mathbb L\rho(u) &= \kappa\int_\mathbb R \mathfrak{p}(u,v)\big[\rho(v) - \rho(u)\big] \,\d v\\
&= \kappa\int_0^\infty \mathfrak p_*(w)\big[\rho(u+w) + \rho(u-w) - 2\rho(u)\big] \,\d w.
\end{align*}
In particular, for $\mathfrak {p}_\mathrm{sy}$ given in Example \ref{exa:translation-invariant}, $\mathbb L=-(-\Delta)^{\frac\gamma2}$ becomes the usual fractional Laplacian operator, and \eqref{eq:hydrodynamics} reads
\[
\partial_t\rho_t(u)=-(-\Delta)^{\frac\gamma2}\rho_t(u), \qquad \rho_0(u)=\rho_\circ(u), \qquad (t,x)\in\mathbb R_+\times\mathbb R,
\]
In this special case, the hydrodynamic limit was proved in \cite[Theorem 2.3]{J08}.
\end{remark}

Let $\pi_t^N=\pi^N(\eta_{t})$ be the empirical distribution of $\eta_t$ and denote the corresponding push-forward measure of $\mathbb P_{\mu_N}$ on $\mathbb D(\mathbb R_+;\mathcal M)$ by $\mathcal{Q}^N$.
Also recall the continuous subspace $C(\mathbb R_+;\mathcal M_\kappa)$, where $\mathcal M_\kappa=\mathcal M_\kappa(\mathbb R)$ is defined in \eqref{eq:meas-space}.
For any path $\{\pi_t\}_{t\ge0}$ in $C(\mathbb R_+;\mathcal M_\kappa)$, there is a density $\rho=\rho(t,u)\in[0,\kappa]$ such that $\pi_t=\rho(t,u)\d u$ for all $t\ge0$.
From the continuity of the path,
\begin{equation}\label{eq:sol-space}
\rho \in L^\infty(\mathbb R_+\times\mathbb R) \cap C\big(\mathbb R_+;L_\mathrm{loc}^1(\mathbb R)\big),
\end{equation}
where $L_\mathrm{loc}^1(\mathbb R)$ is endowed with the weak topology: a sequence $\rho^n$ converges to $\rho$ if and only if $\langle \rho^n,g \rangle \to \langle \rho,g \rangle$ for any $g \in C_c(\mathbb R)$.

\begin{theorem}\label{theo:hydrodynamics}
Assume conditions (A1)--(A3), \eqref{eq:p-sum}--\eqref{eq:tail-2} and \eqref{eq:initial}.
The sequence $\{\mathcal Q^N\}_{N\ge1}$ is tight in $\mathbb D(\mathbb R_+;\mathcal M)$ and any weak limit point $\mathcal Q$ is concentrated on $C(\mathbb R_+;\mathcal M_\kappa)$, with the density $\rho$ being a weak solution to \eqref{eq:hydrodynamics} in the sense of Definition \ref{def:weak-sol}.
\end{theorem}

\subsection{Proof of Theorem \ref{theo:hydrodynamics}}

The proof has the classical structure 
introduced in \cite{GPV} that consists of showing the tightness and identifying the limit point.
We remark that the proofs of all the lemmas in this section are postponed to Section \ref{sec:proof_lemmas}.
In the first lemma, we show the tightness and then characterize the minimal regularity of the limit points.

\begin{lemma}[Tightness]\label{lem:tight}
The sequence $\{\mathcal{Q}^N\}_{N\ge1}$ is tight and any weak limit point $\mathcal{Q}$ is concentrated on $C(\mathbb R_+;\mathcal M_\kappa)$.
In particular, with full $\mathcal Q$-probability, $\pi_t=\rho_t(u)\d u$ such that $0\le\rho_t(u)\le\kappa$ for all $t\ge0$ and almost every $u\in\mathbb R$.
\end{lemma}

Hereafter, assume Lemma \ref{lem:tight} and take a sub-sequential limit point $\mathcal{Q}$ of $\{\mathcal Q^N\}_{N\ge1}$.
Assume, for simplicity, that the whole sequence converges weakly to $\mathcal{Q}$.
By Lemma \ref{lem:tight}, one can identify $\mathcal{Q}$ with the probability measure on the density $\rho$ in the space given by \eqref{eq:sol-space}.
Our aim is to characterize, as much as possible, the density $\rho$ under $\mathcal Q$.
To this end, fix any $G \in C_c^{1,\infty}(\mathbb R_+\times \mathbb R)$ and note that
\begin{equation}\label{eq:martingale}
M_t^N(G) := \langle \pi_t^N,G_t \rangle - \langle \pi_0^N,G_0 \rangle - \int_0^t \big[\langle \pi_s^N,\partial_sG_s \rangle + N^\gamma\mathcal L \langle \pi_s^N, G_s \rangle\big]\,\d s
\end{equation}
defines a martingale under $\mathcal Q^N$ with respect to the natural filtration of the process.
Suppose that the support of $G$ is contained in $[0,T)\times(-R,R)$ for $T$, $R>0$. Now, for any $\eta\in\Omega$,
\begin{equation}
N^\gamma\mathcal L \langle \pi^N(\eta), G_t \rangle = \frac1{N^2}\sum_{(x,y) \in NA_R} H_{t,x,y}^N  \eta(x)\big(\kappa - \eta(y)\big),
\end{equation}
where $A_R:=\mathbb R^2\setminus\{(u,v)\;:\;|u|,|v| \ge R\}$ and $H_{t,x,y}^N=H_t(\tfrac xN,\tfrac yN)$ with 
\begin{equation}\label{eq:h}
H_t(u,v) := \mathfrak{p}(u,v)\big[G_t(v)-G_t(u)\big].
\end{equation}

Our aim is to obtain \eqref{eq:weak-sol} from the limit of \eqref{eq:martingale}.
To do that, we need several intermediate lemmas.
Recall that $\mathbb P_{\mu_N}$ is the distribution of $\{\eta_t\}_{t\ge0}$ in $\mathbb D(\mathbb R_+;\mathcal M)$.

\begin{lemma}\label{lem:martingale}
For any $\delta>0$,
\begin{equation}
\lim_{N\uparrow\infty} \mathbb P_{\mu_N} \left( \sup_{t\in[0,T]} |M_t^N(G)| > \delta \right) = 0.
\end{equation}
\end{lemma}

To make our argument more general, we state the next two lemmas for an arbitrary sequence of probability distributions $\mathbb P_N$ on $\mathbb D(\mathbb{R}_+;\mathbb N_0^\mathbb Z)$ such that
\begin{equation}\label{eq:bound-l2}
\limsup_{N\uparrow\infty} \left\{ \sup_{x\in\mathbb Z}\,\mathbb E_N \left[ \int_0^T \big|\eta_t(x)\big|^2\,\d t \right] \right\} < \infty, \qquad \forall\,T>0,
\end{equation}
where $\mathbb E_N$ denotes the expectation w.r.t.~$\mathbb P_N$.
As the configurations of the generalized exclusion process are uniformly bounded by $\kappa$, $\mathbb P_{\mu_N}$ satisfies the aforementioned condition in a trivial way, and the lemmas below hold without the time integral and uniformly in $\eta\in\Omega$.

\begin{lemma}\label{lem:cutting}
Assume \eqref{eq:bound-l2}.
For any $R>0$,
\begin{equation}\label{eq:cutting}
\limsup_{\epsilon'\downarrow0,N\uparrow\infty} \mathbb E_N \left[ \frac1{N^2}\int_0^T \sum_{(x,y) \in N(A_R \setminus A_{R,\epsilon'})} H_{t,x,y}^N  \eta_t(x)\big(\kappa-\eta_t(y)\big) \,\d t \right] = 0,
\end{equation}
where $A_{R,\epsilon'} := \{(u,v) \in A_R\;:\;\epsilon' < |u-v| < (\epsilon')^{-1}\}$.
\end{lemma}

The next lemma shows that the current field associated to medium-distance jumps can be replaced by a function of the macroscopic density profile.
It is the key step in the characterization of the limit points.

\begin{lemma}\label{lem:medium-replace}
Assume \eqref{eq:bound-l2}.
For any $R>0$ and $\epsilon'>0$,
\begin{equation}\label{eq:medium-replace}
\begin{aligned}
\limsup_{\epsilon\downarrow0,N\uparrow\infty} \mathbb E_N &\left[ \frac1{N^2}\int_0^T \left| \sum_{(x,y) \in NA_{R,\epsilon'}} H_{t,x,y}^N \cdot \Big(\eta_t(x)\big(\kappa - \eta_t(y)\big) \right.\right.\\
&\qquad\left.\left.\vphantom{\sum_{(x,y) \in NA_{R,\epsilon'}}} - \eta_t(x)^{\epsilon N}\big(\kappa - \eta_t(y)^{\epsilon N}\big)\Big) \right| \,\d t \right] = 0.
\end{aligned}
\end{equation}
\end{lemma}

Now, we prove Theorem \ref{theo:hydrodynamics} based on the above lemmas.

\begin{proof}[Proof of Theorem \ref{theo:hydrodynamics}]
Recall that $G$ is an arbitrary test function in $C_c^{1,\infty}(\mathbb R_+\times \mathbb R)$ whose support is contained in $[0,T)\times(-R,R)$.
By standard density arguments, we only need to show for any $\delta>0$ that
\begin{equation}
\begin{aligned}
\mathcal{Q} &\left( \left\vert \langle \rho_\circ, G_0\rangle + \int_0^T\langle \rho_t, \partial_tG_t \rangle \,\d t \right.\right.\\
&\quad\left.\left. + \int_0^T \d t \iint_{A_R} H_t(u,v) \rho_t(u)\big(\kappa - \rho_t(v)\big) \,\d{v}\,\d{u} \right\vert > \delta \right) = 0,
\end{aligned}
\end{equation}
where $H_t(u,v)$ is the function introduced in \eqref{eq:h}.
Notice that 
\[
|H_t(u,v)| \lesssim \mathfrak{p}(u,v)\min\{|u-v|,1\}.
\]
Condition (A3) then yields that it suffices to show
\begin{equation}\label{eq:tobeshown}
\begin{aligned}
\limsup_{\epsilon'\downarrow0} \mathcal{Q} &\left( \left\vert \vphantom{\iint_{A_{R,\epsilon'}}} \langle \rho_\circ, G_0\rangle + \int_0^T\langle \rho_t, \partial_tG_t \rangle \,\d t\right.\right.\\
&\quad\left.\left. + \int_0^T \d t \iint_{A_{R,\epsilon'}} H_t(u,v) \rho_t(u)\big(\kappa-\rho_t(v)\big) \,\d{v}\,\d{u}\right\vert > \delta \right) = 0.
\end{aligned}
\end{equation}
For each fixed $t\ge0$, from the Lebesgue differentiation theorem
\begin{equation}
\rho_t(u) = \lim_{\epsilon\downarrow 0} \langle \pi_t, I_u^{\epsilon}\rangle, \qquad \text{a.e.},
\end{equation}
where $I_u^{\epsilon} := (2\epsilon)^{-1}\mathds{1}_{[u-\epsilon, u+\epsilon]}$.
Since $\rho_t$ is bounded, the convergence above also holds in $L_\mathrm{loc}^2(\mathbb R)$.
Then, \eqref{eq:tobeshown} follows as soon as
\begin{equation}\label{eq:tobeshown2}
\begin{aligned}
\limsup_{\epsilon'\downarrow0,\epsilon\downarrow 0} \mathcal{Q}&\left( \left\vert \vphantom{\iint_{A_{R,\epsilon'}}} \langle \pi_0, G_0 \rangle + \int_0^T \langle \pi_t, \partial_tG_t \rangle\,\d t \right.\right.\\
&\quad\left.\left.+\int_0^T \d t\iint_{A_{R,\epsilon'}} H_t(u,v) \langle \pi_t, I_u^\epsilon \rangle\big(\kappa-\langle\pi_t, I_v^\epsilon \rangle\big) \,\d{v}\,\d{u} \right\vert > \delta \right) = 0.
\end{aligned}
\end{equation}
As $\{\mathcal{Q}^N\}_{N\geq 1}$ converges weakly to $\mathcal{Q}$, \eqref{eq:tobeshown2} holds if\footnote{Here we can apply the Portmanteau Theorem since $I_u^\epsilon$ can be approximated by continuous functions in $L^2(\mathbb R)$.}
\begin{equation}\label{eq:tobeshown3}
\begin{aligned}
\limsup_{\epsilon'\downarrow 0,\epsilon\downarrow 0,N\uparrow\infty} \mathcal{Q}^N &\left( \left\vert \vphantom{\iint_{A_{R,\epsilon'}}} \langle \pi_0^N, G_0 \rangle + \int_0^T \langle \pi_t^N, \partial_tG_t \rangle\,\d t \right.\right.\\
&\quad\left.\left.+\int_0^T \d t\iint_{A_{R,\epsilon'}}  H_t(u,v) \langle \pi_t^N, I_u^\epsilon \rangle\big(\kappa-\langle \pi_t^N, I_v^\epsilon \rangle\big) \,\d{v}\,\d{u} \right\vert > \delta \right) = 0.
\end{aligned}
\end{equation}
Furthermore, note that for any $t\ge0$ and $\epsilon>0$,
\begin{equation}
\lim_{N\to\infty} \sup_{u\in[\frac xN,\frac{x+1}N)} \left| \big\langle \pi_t^N,I_u^\epsilon \big\rangle - \eta_t(x)^{\epsilon N} \right| = 0 \qquad \text{uniformly in $x\in\mathbb Z$}.
\end{equation}
Observe that we may replace the integral on $A_{R,\epsilon'}$ by the corresponding Riemann sum with an error that vanishes as $N\to\infty$.
Hence, \eqref{eq:tobeshown3} is reduced to
\begin{equation}
\begin{aligned}
\limsup_{\epsilon'\downarrow 0,\epsilon\downarrow 0,N\uparrow\infty} \mathcal{Q}^N &\left( \left\vert \langle \pi_0^N, G_0 \rangle + \int_0^T \langle \pi_t^N, \partial_tG_t \rangle\,\d t \vphantom{\sum_{(x,y) \in NA_{R,\delta}}} \right.\right.\\
&\quad\left.\left.+ \frac1{N^2}\int_0^T \sum_{(x,y) \in NA_{R,\epsilon'}} H_{t,x,y}^N  \eta_t(x)^{\epsilon N}\big(\kappa-\eta_t(y)^{\epsilon N}\big) \,\d t \right\vert > \delta \right) = 0.
\end{aligned}
\end{equation}
This follows from \eqref{eq:martingale}, Lemma \ref{lem:martingale}, and Lemmas \ref{lem:cutting} and \ref{lem:medium-replace} applied to $\mathbb P_{\mu_N}$.
\end{proof}

\subsection{Proof of the lemmas}\label{sec:proof_lemmas}

\begin{proof}[Proof of Lemma \ref{lem:tight}]
By \cite[Chapter 4, Theorem 1.3 and Proposition 1.7]{KL99}, tightness follows from the fact that for any $g \in C_c^\infty(\mathbb R)$ and any $\delta>0$ that
\begin{align}
\label{eq:tight-pf-1}
&\sup_{N\ge1} \mathcal Q_N \left\{ \sup_{t\ge0} \big|\langle \pi_t^N,g \rangle\big| \le \kappa\int_\mathbb R |g(u)|\,\d u \right\} = 1,\\
\label{eq:tight-pf-2}
&\limsup_{\epsilon\downarrow0,N\uparrow\infty} \mathcal Q_N \left\{ \sup_{|t-s|<\varepsilon} \big|\langle \pi_t^N,g \rangle - \langle \pi_s^N,g \rangle\big| > \delta \right\} = 0.
\end{align}
Criteria \eqref{eq:tight-pf-1} holds since the occupation variables are bounded by $\kappa$, while \eqref{eq:tight-pf-2} can be shown with the standard method, see, e.g.~\cite[Proposition 4.7]{SS18}.
We omit the details here.
Moreover, \eqref{eq:tight-pf-1} and \eqref{eq:tight-pf-2} also yield that any limit point is concentrated on the paths $\{\pi_t\}_{t\ge0}$ which are continuous in $t$ and satisfy
\begin{equation}
\big|\langle \pi_t,g \rangle\big| \le \kappa\int_\mathbb R |g(u)|\,\d u, \qquad \forall\,t\ge0.
\end{equation}
Therefore, $\pi_t(du)=\rho(t,u)\d u$ with density $|\rho(t,u)|\le\kappa$ for each $t\ge0$ and almost every $u\in\mathbb R$.
As $\rho$ is non-negative, the lemma follows.
\end{proof}

\begin{proof}[Proof of Lemma \ref{lem:martingale}]
The quadratic variation of the martingale is given by
\begin{align*}
\langle M^N (G)\rangle_t &= N^\gamma\int_0^t \mathcal L^N \big[\langle \pi_s^N, G_s \rangle^2\big]-2\langle \pi_s^N, G_s \rangle\mathcal L^N \langle \pi_s^N, G_s \rangle\,\d s\\
&= \frac{1}{N} \int_0^t \sum_{(x,y)\in\mathbb Z^2} \mathfrak{p}\left(\frac{x}{N},\frac{y}{N}\right)\eta(x)\big(\kappa-\eta(y)\big) \left[ G_s \left( \frac yN \right) - G_s \left( \frac xN \right) \right]^2\,\d s.
\end{align*}
Recalling that $G$ is compactly supported and smooth, the lemma follows from Doob's inequality.
We leave the missing details to the reader.
\end{proof}

\begin{proof}[Proof of Lemma \ref{lem:cutting}]
Since $\big|\eta_t(x)\big(\kappa - \eta_t(y)\big)\big| \lesssim 1+\eta_t(x)^2+\eta_t(y)^2$ and
\begin{equation}
\big| H_{t,x,y}^N \big| \lesssim_G  \mathfrak p \left( \frac xN,\frac yN \right)  \min \left\{ \frac{|y-x|}N,1 \right\},
\end{equation}
the expression inside the expectation can be bounded from above by a constant times
\begin{equation}
\begin{aligned}
\frac1{N^2} \sum_{\substack{(x,y) \in NA_R\\|y-x|\le\epsilon'N}} \mathfrak p \left( \frac xN,\frac yN \right) \frac{|y-x|}N + \frac1{N^2} \sum_{\substack{(x,y) \in NA_R\\|y-x|\ge\frac1{\epsilon'}N}} \mathfrak p \left( \frac xN,\frac yN \right).
\end{aligned}
\end{equation}
Using condition (A2), we get for $\epsilon'$ small enough the further upper bound
\begin{equation}
\begin{aligned}
\frac{2N^\gamma}{N^2} \sum_{\substack{|x|<RN\\|y-x|\le\epsilon'N}} \mathfrak{s}(x,y)|y-x| + \frac{2N^\gamma}N \sum_{\substack{|x|<RN\\|y|\ge(\frac1{\epsilon'}-R)N}} \mathfrak{s}(x,y).
\end{aligned}
\end{equation}
The limit then holds due to \eqref{eq:tail-1} and \eqref{eq:tail-2}.
\end{proof}
%

\begin{proof}[Proof of Lemma \ref{lem:medium-replace}]
Recall that $\Gamma_\ell=\{w\in\mathbb Z\;:\;|w|\le\ell\}$.
Since $\mathsf{s}(\kappa - \mathsf{s}')$ is affine both in $\mathsf{s}$ and $\mathsf{s}'$, the sum on the left-hand side of \eqref{eq:medium-replace} is equal to
\begin{equation}\label{eq:using_affine}
\frac1{|\Gamma_{\epsilon N}|^2} \sum_{(x,y) \in NA_{R,\epsilon'}} H_{t,x,y}^N \sum_{z,z'\in\Gamma_{\epsilon N}} \Big[\eta_t(x)\big(\kappa - \eta_t(y)\big) - \eta_t(x+z)\big(\kappa - \eta_t(y+z')\big)\Big].
\end{equation}
Denote by $N\Delta_{R,\epsilon'}(z,z')$ the symmetric difference between $NA_{R,\epsilon'}$ and $NA_{R,\epsilon'}+(z,z')$.
Performing the change of variable $(x,y)\mapsto (x-z,y-z)$, using \eqref{eq:bound-l2} and the expression above, we can bound the left-hand side of \eqref{eq:medium-replace} by $C_T(\mathcal I_1+\mathcal I_2)$, where
\begin{equation}
\begin{aligned}
\mathcal I_1 &:= \limsup_{\epsilon\downarrow0,N\uparrow\infty} \left\{ \sup_{t\in[0,T]} \sup_{z,z'\in\Gamma_{\epsilon N}} \frac1{N^2}\sum_{(x,y) \in NA_{R,\epsilon'}} \big|H_{t,x,y}^N - H_{t,x-z,y-z'}^N\big| \right\};\\
\mathcal I_2 &:= \limsup_{\epsilon\downarrow0,N\uparrow\infty} \left\{ \sup_{t\in[0,T]} \sup_{z,z'\in\Gamma_{\epsilon N}} \frac1{N^2}\sum_{(x,y) \in N\Delta_{R,\epsilon'}(z,z')} \big|H_{t,x-z,y-z'}^N\big| \right\}.
\end{aligned}
\end{equation}
Let $\|G\|_\infty$ denote the uniform norm of $G$. 
When $\epsilon$ is sufficiently small, from the definition of $H_{t,x,y}^N$ in \eqref{eq:h},
\begin{equation}
\big|H_{t,x-z,y-z'}^N\big| \lesssim \sup \left\{ \mathfrak{p}(u,v)\;:\;(u,v) \in A_{R+1,\frac{\epsilon'}2} \right\} \Vert G\Vert_\infty,
\end{equation}
for all $z$, $z'\in\Gamma_{\epsilon N}$, $(x,y) \in N\Delta_{R,\epsilon'}(z,z')$ and $t\in[0,T]$, so that
\begin{equation}
\mathcal I_2 \lesssim_{R,\epsilon'} \limsup_{\epsilon\downarrow0,N\uparrow\infty} \left\{ \frac1{N^2}\sup_{z,z'\in\Gamma_{\epsilon N}} \#\big(N\Delta_{R,\epsilon'}(z,z')\big) \right\} = 0.
\end{equation}
For $\mathcal I_1$, recall the condition (A1).
Then,
\begin{equation}
\begin{aligned}
\big|H_{t,x,y}^N-H_{t,x-z,y-z'}^N\big| &\lesssim \mathfrak p \left( \frac xN,\frac yN \right) \epsilon\|\partial_uG\|_\infty\\
&\qquad + \left| \mathfrak p \left( \frac xN,\frac yN \right) - \mathfrak p \left( \frac{x-z}N,\frac{y-z'}N \right) \right| \|G\|_\infty\\
&\lesssim \sup_{A_{R,\epsilon'}} \mathfrak{p}(u,v) \epsilon\|\partial_uG\|_\infty + \epsilon L_{R,\epsilon'} \|G\|_\infty,
\end{aligned}
\end{equation}
for all $z$, $z'\in\Gamma_{\epsilon N}$, $(x,y) \in NA_{R,\epsilon'}$ and $t\in[0,T]$, where $L_{R,\epsilon'}$ is the Lipschitz constant of $p$ on a small enlargement of $A_{R,\epsilon'}$ that is still compact in $\mathbb R^2_0$.
Therefore,
\begin{equation}
\mathcal I_1 \lesssim_{R,\epsilon'} \limsup_{\epsilon\downarrow0} \left\{ \epsilon \limsup_{N\uparrow\infty} \frac1{N^2}\#\big(NA_{R,\epsilon'}\big) \right\} = 0.
\end{equation}
We then conclude the proof.
\end{proof}

\section{Generalized exclusion with boundary reservoirs}\label{sec:gen_exclusion_boundary}

For the entirety of this section, we consider the spatial domain $I = [0,1]$ and refer to its microscopic pendant $\Lambda_N = \{1, \dots, N-1\}$ as the \emph{bulk}.
Recall that $\Omega_N = \{0,\dots, \kappa\}^{\Lambda_N}$.

The particles move inside the bulk as the generalized exclusion restricted to $\Lambda_N$, with generator acting on $f:\Omega_N\to\mathbb R$ as
\begin{equation}
\mathcal L_\mathrm{bulk}^Nf(\eta) = \sum_{x,y\in\Lambda_N} \mathfrak{p}(x,y)\eta(x)\big(\kappa-\eta(y)\big)  \Big( f(\eta^{x,y}) - f(\eta) \Big).
\end{equation}
The remaining sites $\{z\in\mathbb Z\;:\;z \le 0\}$ (resp.~$\{z\in\mathbb Z\;:\;x \ge N\}$) are called the left (resp.~right) reservoirs.
The interaction between sites in the bulk and the reservoirs is generated by $\mathcal L_\mathrm{bd}^N$ acting on $f:\Omega_N\to\mathbb R$ as
\begin{equation}
\begin{aligned}
\mathcal L_\mathrm{bd}^Nf(\eta) = N^\theta\sum_{x\in\Lambda_N} &\left[ \big(a_-r_{-,x}+a_+r_{+,x}^N\big) \big(\kappa-\eta(x)\big) \Big(f(\eta^{x,+})-f(\eta)\Big) \right.\\
&\left. +\big(b_-r_{x,-}+b_+r_{x,+}^N\big) \eta(x) \Big(f(\eta^{x,-})-f(\eta)\Big) \right],
\end{aligned}
\end{equation}
where $\theta\le0$, $a_\pm$, $b_\pm$ are non-negative constants, and
\begin{equation}\label{eq:bd-rate-1}
\begin{aligned}
&r_{-,x}:=\sum_{y \le 0} \mathfrak{p}(y,x), \qquad r_{x,-}:=\sum_{y \le 0} \mathfrak{p}(x,y),\\
&r_{+,x}^N:=\sum_{y \ge N} \mathfrak{p}(y,x), \qquad r_{x,+}^N:=\sum_{y \ge N} \mathfrak{p}(x,y).
\end{aligned}
\end{equation}
When $\theta<0$, the factor $N^\theta$ regulates the strength of reservoir dynamics by slowing down the exchange of particles between the bulk system and the reservoirs.


\subsection{Assumptions on the jump kernel}

As in Section \ref{sec:gen-ex}, we assume that the jump kernel $\mathfrak{p}=\mathfrak{p}(x,y)$ can be extended to a function on $\mathbb R_0^2$ satisfying the conditions (A1) and (A2).
Recall from \eqref{eq:sym_p} the symmetric part $\mathfrak{s}$ of $\mathfrak p$ and also define its anti-symmetric part as
\begin{equation}
\mathfrak a(u,v) := \frac{\mathfrak{p}(u,v)-\mathfrak{p}(v,u)}2 = -\mathfrak a(v,u) .
\end{equation}
Throughout this section, we assume $\gamma\in(0,1)$ and the following growth condition: there are universal constants $c_1$, $c_2>0$, such that
\begin{equation}\label{eq:growth}
\frac{c_1}{|u-v|^{1+\gamma}} \le |\mathfrak{a}(u,v)| \le \mathfrak{s}(u,v) \le \frac{c_2}{|u-v|^{1+\gamma}}, \qquad \forall\,(u,v)\in\mathbb R_0^2.
\end{equation}
In particular, $|\mathfrak p(u,v)|\lesssim|u-v|^{-1-\gamma}$, and the rates in \eqref{eq:bd-rate-1} are finite.

For each site $x\in\Lambda_N$, define
\begin{equation}\label{eq:bd-rate-2}
\begin{aligned}
\alpha^N(x) := N^\gamma\big(a_- r_{-,x} + a_+ r_{+,x}^N\big), \qquad \beta^N(x) := N^\gamma\big(b_- r_{x,-} + b_+ r_{x,+}^N\big).
\end{aligned}
\end{equation}
From condition (A2) and \eqref{eq:bd-rate-1},
\begin{equation}
\alpha^N(x) = \frac{a_-}N \sum_{y \le 0} \mathfrak p \left( \frac yN,\frac xN \right) + \frac{a_+}N \sum_{y \ge N} \mathfrak p \left( \frac yN,\frac xN \right).
\end{equation}
The growth condition \eqref{eq:growth} then yields the upper bound
\begin{equation}\label{eq:bd-rate-estimate}
\begin{aligned}
\alpha^N(x) &\lesssim a_-\int_{-\infty}^0 \left| v-\frac xN \right|^{-1-\gamma} \,\d v + a_+\int_1^\infty \left| v-\frac xN \right|^{-1-\gamma} \,\d v\\
&\lesssim \frac{a_-}\gamma \left( \frac xN \right)^{-\gamma} + \frac{a_+}\gamma \left( 1-\frac xN \right)^{-\gamma}.
\end{aligned}
\end{equation}
A similar bound holds for $\beta^N(x)$ with $a_\pm$ replaced by $b_\pm$.
When $\theta=0$, we assume further that, for each $u\in(0,1)$,
\begin{equation}\label{eq:tail-bd}
\begin{aligned}
\lim_{N\uparrow\infty} \alpha^N\big([Nu]\big) = a_-\int_{-\infty}^0 \mathfrak p(v,u) \,\d v + a_+\int_1^\infty \mathfrak p(v,u) \,\d v := \alpha(u),\\
\lim_{N\uparrow\infty} \beta^N\big([Nu]\big) = b_-\int_{-\infty}^0 \mathfrak p(u,v) \,\d v + b_+\int_1^\infty \mathfrak p(u,v) \,\d v := \beta(u).
\end{aligned}
\end{equation}
Then, $\alpha(u)$, $\beta(u) \lesssim u^{-\gamma}+(1-u)^{-\gamma}$.
Since $\gamma\in(0,1)$, by the dominated convergence theorem, $\alpha$, $\beta \in C((0,1))$ and the convergence above holds in $L^1$.

\begin{remark}
Recall the totally asymmetric $\kappa$-exclusion in Example \ref{exa:translation-invariant}, where $\mathfrak p_\mathrm{asy}(x,y)=\mathds1_{x<y}|x-y|^{-1-\gamma}$, $\gamma\in(0,1)$.
It is easy to see that the growth condition \eqref{eq:growth} is fulfilled, and
\[
\alpha_\mathrm{asy}(u)=a_-\gamma^{-1}u^{-\gamma}, \qquad \beta_\mathrm{asy}(u)=b_+\gamma^{-1}(1-u)^{-\gamma}.
\]
Similar reservoirs are discussed in \cite{BGJ21,BCGS22,S21}.
\end{remark}

\subsection{Hydrodynamic limit}

Denote by $\eta^N=\{\eta_t^N\}_{t\ge0}$ the Markov process generated by $N^\gamma\mathcal L^N$ and initial distribution $\mu_N$ on $\Omega_N$, where $\mathcal L^N:=\mathcal L_\mathrm{bulk}^N+\mathcal L_\mathrm{bd}^N$.
Let $\mathbb P_{\mu_N}$ be its law on the path space $\mathbb D(\mathbb{R}_+;\Omega_N)$ and $\mathbb E_{\mu_N}$ be the corresponding expectation.
The superscript $N$ in $\eta_t^N$ will be frequently omitted when no confusion is caused.

Similarly to \eqref{em_meas}, the empirical distribution of $\eta\in\Omega_N$ is given by
\begin{equation}\label{em_meas-bd}
\pi^N(\eta) := \dfrac{1}{N}\sum_{x\in\Lambda_N} \eta(x)\delta_{x/N}.
\end{equation}
Observe that $\pi^N(\eta)\in\mathcal M=\mathcal M([0,1])$.

Assume the existence of a measurable profile $\rho_\circ: [0,1]\to[0,\kappa]$, such that (cf.~\eqref{eq:initial}) for any $g \in C([0,1])$ and $\delta > 0$,
\begin{equation}\label{eq:initial-bd}
\lim_{N\to+\infty} \mu_N\Big( \left\{ \eta\in\Omega_N\;:\; \left\vert \big\langle \pi^N(\eta),g \big\rangle - \langle \rho_\circ, g \rangle \right\vert > \delta \right\} \Big) = 0.
\end{equation}

For $\rho \in L^\infty((0,1))$, let $\mathbb L\rho:C_c^\infty([0,1])\to\mathbb R$ be a bounded linear functional, such that for all $g \in C_c^\infty(\mathbb R)$, $(\mathbb L\rho)(g)=\langle \mathbb L\rho,g \rangle$ is given by
\begin{equation}\label{eq:operator-bd_L}
\langle \mathbb L\rho,g \rangle := \int_0^1 \int_0^1 \mathfrak{p}(u,v)\big[g(v)-g(u)\big]\rho(u)\big(\kappa-\rho(v)\big) \,\d v\,\d u.
\end{equation}
Observe that \eqref{eq:operator-bd_L} is similar to the definition in \eqref{eq:operator_L} with the only difference being at the level of the domain of integration, since the macroscopic space is now given as the interval $(0,1)$.
Let $V=V(u,\rho)$ be the source function given by
\begin{equation}
V(u,\rho) := \kappa\alpha(u)-\big[\alpha(u)+\beta(u)\big]\rho, \qquad \forall\,(u,\rho)\in(0,1)\times[0,\kappa].
\end{equation}
Recall that $\theta\le0$.
We expect the hydrodynamic equation
\begin{equation}\label{eq:hydrodynamics-bd}
\partial_t\rho(t,u) = \mathbb L\rho(t,u) + \mathds1_{\theta=0}V(u,\rho(t,u)), \qquad \rho(0,u)=\rho_\circ(u),
\end{equation}
with the boundary conditions specified in the definition below.

\begin{definition}\label{def:weak-sol-bd}
A weak solution to \eqref{eq:hydrodynamics-bd} is a measurable function $\rho :\mathbb R_+\times[0,1]\to [0,\kappa]$, such that for all test functions $G \in C_c^{1,\infty}(\mathbb R_+\times[0,1])$,
\begin{equation}\label{eq:weak-sol-bd-1}
\langle \rho_0,G_0 \rangle + \int_0^\infty \langle \rho_t, \partial_tG_t \rangle \,\d t + \int_0^\infty \big\langle \mathds1_{\theta=0}V(\cdot,\rho_t)+\mathbb L\rho_t, G_t \big\rangle \,\d t = 0
\end{equation}
and, for all $0<s<t$, it verifies the following asymptotic formulae as $\epsilon\downarrow0$:
\begin{align}
\label{eq:weak-sol-bd-2}
\begin{aligned}
\int_s^t \int_0^\epsilon \left\{ \int_\epsilon^1 \Big[\mathfrak{p}(v,u)\rho_r(v)\big(\kappa-\rho_r(u)\big) - \mathfrak{p}(u,v)\rho_r(u)\big(\kappa-\rho_r(v)\big)\Big] \,\d v \right.\qquad\qquad\quad\\
+ \left. \vphantom{\int_\epsilon^1} \mathds1_{\theta=0} V\big(u,\rho_r(u)\big) \right\} \,\d u\,\d r = O(\epsilon),
\end{aligned}
\\
\label{eq:weak-sol-bd-3}
\begin{aligned}
\int_s^t \int_{1-\epsilon}^1 \left\{ \int_0^{1-\epsilon} \Big[\mathfrak{p}(v,u)\rho_r(v)\big(\kappa-\rho_r(u)\big) - \mathfrak{p}(u,v)\rho_r(u)\big(\kappa-\rho_r(v)\big)\Big] \,\d v \right.\qquad\qquad\quad\\
+ \left. \vphantom{\int_\epsilon^1} \mathds1_{\theta=0}V\big(u,\rho_r(u)\big) \right\} \,\d u\,\d r = O(\epsilon).
\end{aligned}
\end{align}
\end{definition}

As before, denote by $\mathcal Q^N$ the push-forward measure on $\mathbb D(\mathbb{R}_+;\mathcal M)$ of $\mathbb P_{\mu_N}$ under the empirical map $\pi^N$ in \eqref{em_meas-bd}.

\begin{theorem}\label{theo:hydrodynamics-bd}
Assume conditions (A1), (A2), \eqref{eq:growth} and \eqref{eq:initial-bd}.
Assume also \eqref{eq:tail-bd} when $\theta=0$.
The sequence $\{\mathcal Q^N\}_{N\ge1}$ is tight and any weak limit point $\mathcal Q$ is concentrated on the continuous paths $C(\mathbb{R}_+;\mathcal M_\kappa)$.
Moreover, the corresponding density function
\begin{equation}
\rho \in L^\infty\big(\mathbb R_+\times(0,1)\big) \cap C\big(\mathbb R_+;L^1((0,1))\big)
\end{equation}
is $\mathcal Q$-almost surely a weak solution in the sense of Definition \ref{def:weak-sol-bd}, where $L^1((0,1))$ is endowed with the weak topology.
\end{theorem}

\paragraph{Discussion on \eqref{eq:weak-sol-bd-2} and \eqref{eq:weak-sol-bd-3}.}
We heuristically interpret these limits as weak Dirichlet boundary conditions.
Indeed, the integrand of the left-hand side of \eqref{eq:weak-sol-bd-2} is equal to
\begin{align*}
\kappa\mathfrak s(u,v)\big[\rho_r(v)-\rho_r(u)\big] &+ \mathfrak a(u,v)\big(2\rho_r(u)\rho_r(v)-\kappa\rho_r(u)-\kappa\rho_r(v)\big)\\
&+ \mathds1_{\theta=0}\big(\kappa \alpha(u)-\alpha(u)\rho_r(u)-\beta(u)\rho_r(u)\big).
\end{align*}
Let us suppose that $u \mapsto \int_s^t \rho_r(u) \,\d r$ is $\gamma_*$-H\"older continuous with some $\gamma_*>\gamma$, and $(u,v) \mapsto \int_s^t \rho_r(u)\rho_r(v) \,\d r$ is continuous.
From \eqref{eq:growth},
\[
\left| \int_s^t \int_0^\epsilon \int_\epsilon^1 \mathfrak s(u,v)\big(\rho_r(v)-\rho_r(u)\big) \,\d v\,\d u\,\d r \right| \lesssim \int_0^\epsilon \int_\epsilon^1 \frac1{(v-u)^{1+\gamma-\gamma_*}} \,\d v\,\d u = O(\epsilon).
\]
Meanwhile, from the above continuity assumption, as $\epsilon\downarrow0$,
\begin{align*}
&\int_s^t \int_0^\epsilon \int_\epsilon^1 \mathfrak a(u,v)\big(2\rho_r(u)\rho_r(v)-\kappa\rho_r(u)-\kappa\rho_r(v)\big) \,\d v\,\d u\,\d r\\
=&\,\left[ o(1) + 2\int_s^t \big(\rho_r^2(0) - \kappa\rho_r(0)\big) \,\d r \right] \int_0^\epsilon \int_\epsilon^1 \mathfrak a(u,v) \,\d v\,\d u
\end{align*}
and
\begin{align*}
&\int_s^t \int_0^\epsilon \big(\kappa\alpha(u)-\alpha(u)\rho_r(u)-\beta(u)\rho_r(u)\big) \,\d v\,\d u\,\d r\\
=&\;(t-s)\kappa\int_0^\epsilon \alpha(u) \,\d u - \left[ o(1) + \int_s^t \rho_r(0) \,\d r \right] \int_0^\epsilon \big(\alpha(u)+\beta(u)\big) \,\d u.
\end{align*}
Since $\mathfrak a(u,v) \gtrsim |u-v|^{-1-\gamma}$ and $\gamma>0$,
\[
\lim_{\epsilon\downarrow0} \frac1\epsilon \int_0^\epsilon \int_\epsilon^1 \mathfrak a(u,v) \,\d v\,\d u \gtrsim \lim_{\epsilon\downarrow0} \frac1\epsilon \int_0^\epsilon \int_\epsilon^1 \frac1{(v-u)^{1+\gamma}} \,\d v\,\d u = +\infty.
\]
When $\theta=0$, suppose furthermore that the following limits exist:
\begin{align*}
\mathfrak C_1 := \lim_{\epsilon\downarrow0} \frac{\int_0^\epsilon \alpha(u) \,\d u}{\int_0^\epsilon \int_\epsilon^1 \mathfrak a(u,v) \,\d v\,\d u} = a_- \cdot \lim_{\epsilon\downarrow0} \frac{\int_0^\epsilon \int_{-\infty}^0 \mathfrak p(v,u) \,\d v\,\d u}{\int_0^\epsilon \int_\epsilon^1 \mathfrak a(u,v) \,\d v\,\d u},\\
\mathfrak C_2 := \lim_{\epsilon\downarrow0} \frac{\int_0^\epsilon \beta(u) \,\d u}{\int_0^\epsilon \int_\epsilon^1 \mathfrak a(u,v) \,\d v\,\d u} = b_- \cdot \lim_{\epsilon\downarrow0} \frac{\int_0^\epsilon \int_{-\infty}^0 \mathfrak p(u,v) \,\d v\,\d u}{\int_0^\epsilon \int_\epsilon^1 \mathfrak a(u,v) \,\d v\,\d u}.
\end{align*}
By \eqref{eq:growth}, these limits are finite.
Then, \eqref{eq:weak-sol-bd-2} implies that
\[
2\int_s^t \big(\rho_r^2(0)-\kappa\rho_r(0)\big) \,\d r + \int_s^t \big(\kappa\mathfrak C_1-(\mathfrak C_1+\mathfrak C_2) \rho_r(0)\big) \,\d r = 0.
\]
As the above holds for all $s<t$, we have for almost every $r>0$ that
\[
\kappa\mathfrak C_1-(\mathfrak C_1+\mathfrak C_2)\rho_r(0)=2\rho_r(0)\big(\kappa-\rho_r(0)\big).
\]
Similarly, suppose that the following limits exist:
\begin{align*}
\mathfrak C_3 := \lim_{\epsilon\downarrow0} \frac{\int_{1-\epsilon}^1 \alpha(u) \,\d u}{\int_{1-\epsilon}^1 \int_\epsilon^1 \mathfrak a(u,v) \,\d v\,\d u} = a_+ \cdot \lim \frac{\int_{1-\epsilon}^1 \int_1^\infty \mathfrak p(v,u) \,\d v\,\d u}{\int_{1-\epsilon}^1 \int_0^{1-\epsilon} \mathfrak a(u,v) \,\d v\,\d u},\\
\mathfrak C_4 := \lim_{\epsilon\downarrow0} \frac{\int_{1-\epsilon}^1 \beta(u) \,\d u}{\int_{1-\epsilon}^1 \int_\epsilon^1 \mathfrak a(u,v) \,\d v\,\d u} = b_+ \cdot \lim \frac{\int_{1-\epsilon}^1 \int_1^\infty \mathfrak p(u,v) \,\d v\,\d u}{\int_{1-\epsilon}^1 \int_0^{1-\epsilon} \mathfrak a(u,v) \,\d v\,\d u}.
\end{align*}
Then, \eqref{eq:weak-sol-bd-3} implies, for almost every $r>0$, that
\[
\kappa\mathfrak C_3-(\mathfrak C_3+\mathfrak C_4)\rho_r(1)=2\rho_r(1)\big(\kappa-\rho_r(1)\big).
\]
Thanks to \eqref{eq:growth}, the continuous function $\mathfrak a$ does not change its sign in the region $\{u<v\}$.
Hence, we fix w.l.o.g. that $\mathfrak a(u,v)>0$ for $u<v$, i.e. the drift is towards the right.
Then, $\mathfrak C_1$, $\mathfrak C_2\ge0$ and $\mathfrak C_3$, $\mathfrak C_4\le0$.
If $\mathfrak C_j\not=0$, $j=1$, ..., $4$, the roots which lie in $[0,\kappa]$ are
\begin{align*}
\rho_r(0) = \frac14 \left( 2\kappa+\mathfrak C_1+\mathfrak C_2 - \sqrt{(2\kappa+\mathfrak C_1+\mathfrak C_2)^2-8\kappa\mathfrak C_1} \right),\\
\rho_r(1) = \frac14 \left( 2\kappa+\mathfrak C_3+\mathfrak C_4 + \sqrt{(2\kappa+\mathfrak C_3+\mathfrak C_4)^2-8\kappa\mathfrak C_3}\right).
\end{align*}
If some $\mathfrak C_j=0$, these formulae still hold through continuous limits.
When $\theta<0$, it can be viewed as the case in which $\mathfrak C_j$ decay as $O(N^{-\theta})$ to $0$, so
\[
\rho_r(0)=0, \qquad \rho_r(1)=\kappa.
\]
In particular, for the totally asymmetric $\kappa$-exclusion corresponding to $\mathfrak p_\mathrm{asy}(x,y) = \mathds1_{x<y}|x-y|^{-1-\gamma}$ with $\gamma\in(0,1)$, we have \emph{formal} Dirichlet boundary conditions
\[
\rho_r(0) = \min\{a_-,\kappa\}, \qquad \rho_r(1) = \max\{\kappa-b_+,0\}.
\]
See also \cite{X24,X22} for similar boundaries in dynamics with nearest-neibghbor jumps.

\subsection{Proof of Theorem \ref{theo:hydrodynamics-bd}}

The proof is analogous to the one in previous section:
we first show tightness of $\mathcal Q^N$ and that any limit point is concentrated on the desired path space.
Then, the limit point is characterized by a series of lemmas.
The proofs of all the lemmas in this section are postponed to Section \ref{sec:proof_lemmas-bd}.
We begin with the following tightness result that holds for both $\theta<0$ and $\theta=0$.

\begin{lemma}[Tightness]\label{lem:tight-bd}
The sequence $\{\mathcal{Q}^N\}_{N\geq 1}$ is tight in $\mathbb{D}(\mathbb R_+;\mathcal{M})$ and any weak limit point $\mathcal{Q}$ is concentrated on $C(\mathbb{R}_+;\mathcal M_\kappa)$.
\end{lemma}

Hereafter, assume Lemma \ref{lem:tight-bd} to hold and take a sub-sequential limit point $\mathcal Q$ of $\{\mathcal Q^N\}_{N\ge1}$.
Assume for simplicity that the whole sequence converges weakly to $\mathcal Q$.
As before, we identify $\mathcal Q$ with the distribution of the density $\rho$ verifying $\pi_t=\rho(t,u)\d u$.
Recall that $\mathcal L^N=\mathcal L_\mathrm{bulk}^N+\mathcal L_\mathrm{bd}^N$.
For $G \in C_c^{1,\infty}(\mathbb R_+\times[0,1])$ and $\eta\in\Omega_N$,
\begin{equation}
\begin{aligned}
N^\gamma\mathcal L^N \langle \pi^N(\eta), G_t \rangle =\;&\frac1{N^2}\sum_{x,y\in\Lambda_N} H_{t,x,y}^N \eta(x)\big(\kappa -\eta(y)\big)\\
&+ \frac{N^\theta}N \sum_{x\in\Lambda_N} G_t \left( \frac xN \right)  \big[\kappa\alpha^N(x)-(\alpha^N(x)+\beta^N(x))\eta(x)\big].
\end{aligned}
\end{equation}
where $H_{t,x,y}^N=H_t(\tfrac xN,\tfrac yN)$ is given in \eqref{eq:h} and $\alpha^N$, $\beta^N$ are given in \eqref{eq:bd-rate-2}.

\begin{lemma}\label{lem:martingale-bd}
For any $\delta>0$, $\lim_{N\uparrow\infty} \mathbb P_{\mu_N} (\sup_{t\in[0,T]} |M_t^N(G)| > \delta) = 0$, where
\begin{equation}
M_t^N(G) := \langle \pi_t^N,G_t \rangle - \langle \pi_0^N,G_0 \rangle - \int_0^t \big[\langle \pi_s^N,\partial_sG_s \rangle + N^\gamma\mathcal L^N \langle \pi_s^N, G_s \rangle\big]\,\d s.
\end{equation}
\end{lemma}

Similarly to the previous section, the remaining lemmas are stated for a sequence of probability measures $\mathbb P_N$ on $\mathbb D(\mathbb R_+;\mathbb N_0^{\Lambda_N})$, such that (cf.~\eqref{eq:bound-l2})
\begin{equation}\label{eq:bound-l2-bd}
\limsup_{N\uparrow\infty} \sup_{x\in\Lambda_N} \left\{ \mathbb E_N \left[ \int_0^T \big|\eta_t(x)\big|^2\,\d t \right] \right\} < \infty, \qquad \forall\,T>0.
\end{equation}

\begin{lemma}\label{lem:cutting-bd}
For $\epsilon'>0$, let
\begin{equation}
B_{\epsilon'}:=\big\{ (u,v)\in [0,1]^2 \;:\; \epsilon'<u,v<1-\epsilon', |u-v|>\epsilon' \big\}.
\end{equation}
Assume \eqref{eq:bound-l2-bd}.
Then,
\begin{equation}\label{eq:cutting-bd}
\limsup_{\epsilon'\downarrow0,N\uparrow\infty} \mathbb E_N \left[ \frac1{N^2}\int_0^T \sum_{(x,y)\in\Lambda_N^2 \setminus NB_{\epsilon'}} H_{t,x,y}^N \eta_t(x)\big(\kappa-\eta_t(y)\big)\,\d t \right] = 0.
\end{equation}
\end{lemma}

\begin{lemma}\label{lem:medium-replace-bd}
Assume \eqref{eq:bound-l2-bd}.
For any fixed $\epsilon'>0$,
\begin{equation}\label{eq:medium-replace-bd}
\begin{aligned}
\limsup_{\epsilon\downarrow0,N\uparrow\infty} \mathbb E_N &\left[ \frac1{N^2}\int_0^T \left| \sum_{(x,y) \in NB_{\epsilon'}} H_{t,x,y}^N \right.\right.\\
&\quad\left.\left. \vphantom{\sum_{(x,y) \in NB_{\epsilon'}}} \Big(\eta_t(x)\big(\kappa - \eta_t(y)\big) - \eta_t(x)^{\epsilon N}\big(\kappa - \eta_t(y)^{\epsilon N}\big)\Big) \right| \,\d t \right] = 0.
\end{aligned}
\end{equation}
\end{lemma}

When $\theta=0$, the source term is formulated by applying the next lemma.

\begin{lemma}\label{lem:source}
Assume \eqref{eq:bound-l2-bd}.
Then,
\begin{equation}\label{eq:source}
\begin{aligned}
\limsup_{\epsilon'\downarrow0,\epsilon\downarrow0,N\uparrow\infty} \mathbb E_N &\left[ \frac1N \int_0^T \left| \sum_{x=1}^{N-1} G_t \left( \frac xN \right) \alpha^N(x)\eta_t(x) \right.\right.\\
&\qquad\left.\left. - \sum_{x=\epsilon'N}^{(1-\epsilon')N} G_t \left( \frac xN \right) \alpha \left( \frac xN \right) \eta_t(x)^{\epsilon N} \right| \,\d t \right] = 0.
\end{aligned}
\end{equation}
The analogous result holds when $(\alpha^N,\alpha)$ is replaced by $(\beta^N,\beta)$.
\end{lemma}

Finally, the boundary conditions follow from the next lemma.

\begin{lemma}\label{lem:boundary}
Assume \eqref{eq:bound-l2-bd}.
For any fixed $\epsilon'>0$,
\begin{align}
\label{eq:boundary-1}
\begin{aligned}
\limsup_{\epsilon''\downarrow0,\epsilon\downarrow0,N\uparrow\infty} \mathbb E_N &\left[ \frac1{N^2} \int_0^T \left| \sum_{x=1}^{\epsilon'N} \sum_{y=\epsilon'N+1}^{N-1} \mathfrak p \left( \frac xN,\frac yN \right)  \eta(x)\big(\kappa - \eta(y)\big) \right.\right.\\
&\quad\left.\left. - \sum_{x=\epsilon''N}^{\epsilon'N} \sum_{y=(\epsilon'+\epsilon'')N}^{(1-\epsilon'')N} \mathfrak p \left( \frac xN,\frac yN \right)  \eta_t(x)^{\epsilon N}\big(\kappa - \eta_t(y)^{\epsilon N}\big) \right|\,\d t \right] = 0,
\end{aligned}
\\
\label{eq:boundary-2}
\begin{aligned}
\limsup_{\epsilon''\downarrow0,\epsilon\downarrow0,N\uparrow\infty} \mathbb E_N &\left[ \frac1{N^2} \int_0^T \left| \sum_{x=1}^{\epsilon'N} \sum_{y=\epsilon'N+1}^{N-1} \mathfrak p \left(\frac yN,\frac xN \right)  \eta(y)\big(\kappa - \eta(x)\big) \right.\right.\\
&\quad\left.\left. - \sum_{x=\epsilon''N}^{\epsilon'N} \sum_{y=(\epsilon'+\epsilon'')N}^{(1-\epsilon'')N} \mathfrak p \left( \frac yN,\frac xN \right) \eta_t(y)^{\epsilon N}\big(\kappa - \eta_t(x)^{\epsilon N}\big) \right|\,\d t \right] = 0.
\end{aligned}
\end{align}
The analogous statements hold for the right boundary.
\end{lemma}

Now, we show the main theorem based on these lemmas.

\begin{proof}[Proof of Theorem \ref{theo:hydrodynamics-bd}]
To shorten the notation, we write
\begin{equation}
V_\theta:=\mathds1_{\theta=0}V(u,\rho) = \mathds1_{\theta=0}\big(\kappa\alpha(u)-(\alpha(u)+\beta(u))\rho\big).
\end{equation}
Using Lemma \ref{lem:tight-bd}, take $\mathcal Q$ to be a limit point of $\{\mathcal Q^N\}_{N\geq 1}$ with $\rho=\rho(t,u)$ the density of a typical path under $\mathcal Q$.
Let $G \in C^{1,\infty}_c(\mathbb R_+\times\mathbb R)$ be an arbitrary test function and fix $T>0$ sufficiently large such that $G_t\equiv0$ for all $t \ge T$.
We first prove for any $\delta>0$ that
\begin{equation}\label{eq:tobeshown-bd-1}
\begin{aligned}
\mathcal{Q} &\left( \bigg\vert \langle \rho_\circ, G_0\rangle + \int_0^T \langle \rho_t, \partial_tG_t \rangle \,\d t + \int_0^T \langle V_\theta(\cdot,\rho_t), G_t \rangle \,\d t \vphantom{\iint_{[0,1]^2}} \right.\\
&\qquad\left. + \int_0^T \d t \iint_{[0,1]^2} H_t(u,v)  \rho_t(u)\big(\kappa - \rho_t(v)\big)\,\d{v}\,\d{u}\bigg\vert > \delta \right) = 0,
\end{aligned}
\end{equation}
where $H_t(u,v)$ is defined in \eqref{eq:h}.
Using \eqref{eq:growth} and $\gamma<1$, one obtains
\begin{equation}\label{eq:integrability-bd}
\iint_{[0,1]^2} \mathfrak{p}(u,v) |u-v| \,\d v\,\d u \lesssim \iint_{[0,1]^2} |u-v|^{-\gamma} \,\d v\,\d u < +\infty.
\end{equation}
Since $|H_t(u,v)| \lesssim \mathfrak{p}(u,v)|u-v|$ and $\rho$ is bounded, \eqref{eq:integrability-bd} and the integrability of $\alpha$, $\beta$ allow us to show instead
\begin{equation}
\begin{aligned}
\limsup_{\epsilon'\downarrow0} \mathcal{Q} &\left( \bigg\vert \langle \rho_\circ, G_0\rangle + \int_0^T \langle \rho_t, \partial_tG_t \rangle \,\d t + \int_0^T \int_{\epsilon'}^{1-\epsilon'} V_\theta(u,\rho_t)G_t(u)\,\d u\,\d t \right.\\
&\qquad \left. + \int_0^T \d t \iint_{B_{\epsilon'}} H_t(u,v) \rho_t(u)\big(\kappa - \rho_t(v)\big)\,\d{v}\,\d{u}\bigg\vert > \delta \right) = 0.
\end{aligned}
\end{equation}
By the same manipulations as used to prove Theorem \ref{theo:hydrodynamics}, it suffices to show
\begin{equation}
\begin{aligned}
\limsup_{\epsilon'\downarrow 0,\epsilon\downarrow 0,N\uparrow\infty} \mathcal{Q}^N &\left( \bigg\vert \langle \pi_0^N, G_0 \rangle + \int_0^T \langle \pi_t^N, \partial_tG_t \rangle \,\d t \vphantom{\sum_{(x,y) \in B_{N,\epsilon'}}}\right.\\
&\qquad + \frac1N \int_0^T \sum_{x=\epsilon'N}^{(1-\epsilon')N} G_t \left( \frac xN \right) V_\theta \left( \frac xN,\eta_t(x)^{\epsilon N} \right) \,\d t\\
&\qquad\left.+ \frac1{N^2}\int_0^T \sum_{(x,y) \in N B_{\epsilon'}} H_{t,x,y}^N \eta_t(x)^{\epsilon N}\big(\kappa - \eta_t(y)^{\epsilon N}\big)\,\d t\bigg\vert > \delta \right) = 0.
\end{aligned}
\end{equation}
Noting that the second integral in the left-hand side of the previous display is linearly dependent on $\eta_t$, we can conclude from Lemma \ref{lem:martingale-bd} and Lemmas \ref{lem:cutting-bd}, \ref{lem:medium-replace-bd} and  \ref{lem:source} when applied to $\mathbb P_N = \mathbb P_{\mu_N}$.

By \eqref{eq:tobeshown-bd-1}, the density $\rho$ satisfies \eqref{eq:weak-sol-bd-1}, $\mathcal Q$-almost surely.
We then focus on the boundary characterization.
As the two assertions \eqref{eq:weak-sol-bd-2} and \eqref{eq:weak-sol-bd-3} are analogous, we concentrate on the left boundary \eqref{eq:weak-sol-bd-2}.
It suffices to show that, for any $0<s<t$ and $\epsilon'\in(0,1)$ sufficiently small but fixed, that
\begin{equation}\label{eq:tobeshown-bd-2}
\begin{aligned}
\limsup_{\epsilon''\downarrow0,\epsilon\downarrow0,N\uparrow\infty}\! &\mathbb E_{\mu_N}\!\! \left[ \int_s^t \left| \frac1{N^2} \sum_{x=\epsilon''N}^{\epsilon'N} \sum_{y=(\epsilon'+\epsilon'')N}^{(1-\epsilon'')N} \mathfrak p \left( \frac yN,\frac xN \right) \eta_r(y)^{\epsilon N}\big(\kappa - \eta_r(x)^{\epsilon N}\big) \right.\right.\\
&\hspace{2.1cm}- \frac1{N^2} \sum_{x=\epsilon''N}^{\epsilon'N} \sum_{y=(\epsilon'+\epsilon'')N}^{(1-\epsilon'')N} \mathfrak p \left( \frac xN,\frac yN \right) \eta_r(x)^{\epsilon N}\big(\kappa - \eta_r(y)^{\epsilon N}\big)\\
&\hspace{2.1cm}\left.\left. + \frac{N^\theta}N \sum_{x=\epsilon''N}^{\epsilon'N} \left[ \alpha \left( \frac xN \right) \big(\kappa-\eta_r(x)^{\epsilon N} \big) - \beta \left( \frac xN \right) \eta_r(x)^{\epsilon N} \right] \right| \,\d r \right] \lesssim \epsilon'.
\end{aligned}
\end{equation}
Define $F_{\epsilon'}(\eta):=N^{-1}\sum_{x=1}^{\epsilon'N} \eta(x)$ to be the cumulative mass at the boundary region $(0,\epsilon')$.
By Dynkin's formula,
\begin{equation}
\begin{aligned}
M^N_t = M^{N,\epsilon'}_t :=\,&F_{\epsilon'}(\eta_t) - F_{\epsilon'}(\eta_0) - N^\gamma\int_0^t \mathcal{L}^N F_{\epsilon'}(\eta_s)\,\d{s}
\end{aligned}
\end{equation}
is a zero-mean martingale.
Similarly to Lemma \ref{lem:martingale-bd}, one finds that $M^N_t$ has a uniformly vanishing quadratic variation on any finite time horizon as $N\uparrow\infty$.
An application of Doob's Maximal inequality yields that
\begin{equation}
\limsup_{N\uparrow\infty} \mathbb{E}_{\mu_N}\left[ \sup_{r \in [s,t]} \left\vert M_r^{N,\epsilon'}\right\vert\right] = 0
\end{equation}
for any 
$\epsilon'>0$.
The fact that $|F_{\epsilon'}| \le \kappa\epsilon'$ then leads to
\begin{equation}\label{eq:bd-pf-1}
\limsup_{N\uparrow\infty} \mathbb{E}_{\mu_N}\left[\left\vert N^\gamma\int_s^t \mathcal{L}^N F_{\epsilon'}(\eta_r)\,\d r \right\vert\right] \le 2\kappa\epsilon'.
\end{equation}
To relate \eqref{eq:bd-pf-1} to \eqref{eq:tobeshown-bd-2}, note that
\begin{equation}\label{eq:bd-pf-2}
\begin{aligned}
N^\gamma\mathcal{L}^N F_{\epsilon'}(\eta) &= \frac1{N^2}\sum_{x=1}^{\epsilon'N} \sum_{y=\epsilon'N+1}^{N-1} \mathfrak p \left( \frac yN,\frac xN \right)  \eta(y)\big(\kappa - \eta(x)\big)\\
&\quad - \frac1{N^2}\sum_{x=1}^{\epsilon'N} \sum_{y=\epsilon'N+1}^{N-1} \mathfrak p \left( \frac xN,\frac yN \right)  \eta(x)\big(\kappa - \eta(y)\big)\\
&\quad + \frac{N^\theta}N \sum_{x=1}^{\epsilon'N} \big[\kappa\alpha^N(x)-(\alpha^N(x)+\beta^N(x))\eta(x)\big].
\end{aligned}
\end{equation}
Repeating the proof of Lemma \ref{lem:source}, we obtain for every $\epsilon'>0$ that
\begin{equation}\label{eq:bd-pf-3}
\limsup_{\epsilon''\downarrow0,\epsilon\downarrow0,N\uparrow\infty} \mathbb E_{\mu_N} \left[ \frac1N \int_s^t \left| \sum_{x=1}^{\epsilon'N} \alpha^N(x)\eta_r(x) - \sum_{x=\epsilon''N}^{\epsilon'N} \alpha \left( \frac xN \right) \eta_r(x)^{\epsilon N} \right| \,\d r \right] = 0.
\end{equation}
Similar limit holds with $(\alpha^N,\alpha)$ replaced by $(\beta^N,\beta)$.
Then, \eqref{eq:tobeshown-bd-2} is a consequence of \eqref{eq:bd-pf-1}, \eqref{eq:bd-pf-2}, \eqref{eq:bd-pf-3}, and Lemma \ref{lem:boundary} applied to $\mathbb P_{\mu_N}$.
\end{proof}

\subsection{Proof of the lemmas}\label{sec:proof_lemmas-bd}

Lemma \ref{lem:tight-bd} and \ref{lem:martingale-bd} are proved with the same argument as used for the infinite dynamics, see Lemma \ref{lem:tight} and \ref{lem:martingale}.
The details are omitted.

\begin{proof}[Proof of Lemma \ref{lem:cutting-bd}]
Since $|\eta_t(x)(\kappa-\eta_t(y))| \lesssim (1+\eta_t(x)^2+\eta_t(y)^2)$, the assumption \eqref{eq:bound-l2-bd} implies that the left-hand side of \eqref{eq:cutting-bd} is bounded from above by
\begin{equation}
\limsup_{\epsilon'\downarrow0,N\uparrow\infty} \frac1{N^2} \sum_{(x,y)\in\Lambda_N^2 \setminus NB_{\epsilon'}} \mathfrak p \left( \frac xN,\frac yN \right) \frac{|y-x|}N.
\end{equation}
Next, observe that \eqref{eq:growth} yields
\begin{equation}
\begin{aligned}
&\limsup_{\epsilon'\downarrow0,N\uparrow\infty} \frac1{N^2}\sum_{(x,y)\in\Lambda_N^2 \setminus NB_{\epsilon'}} \mathfrak p \left( \frac xN,\frac yN \right) \frac{|y-x|}N\mathds1_{|\frac yN-\frac xN|<\epsilon'}\\
\lesssim&\,\limsup_{\epsilon'\downarrow0} \int_0^1 \int_{u-\epsilon'}^{u+\epsilon'} \frac1{|u-v|^\gamma} \,\d v\,\d u = 0,
\end{aligned}
\end{equation}
where we used the fact that $\gamma<1$.
Meanwhile, by \eqref{eq:integrability-bd}, we have
\begin{equation}
\begin{aligned}
&\limsup_{\epsilon'\downarrow0,N\uparrow\infty} \frac1{N^2}\sum_{(x,y)\in\Lambda_N^2 \setminus NB_{\epsilon'}} \mathfrak p \left( \frac xN,\frac yN \right) \frac{|y-x|}N \mathds1_{|\frac yN-\frac xN|\ge\epsilon'}\\
=&\,\lim_{\epsilon'\downarrow0} \iint_{[0,1]^2\setminus(\epsilon',1-\epsilon')^2} \mathfrak{p}(u,v)|u-v|\mathds1_{|u-v|\ge\epsilon'}\,\d v\,\d u= 0.
\end{aligned}
\end{equation}
This concludes the proof.
\end{proof}

\begin{proof}[Proof of Lemma \ref{lem:medium-replace-bd}]
Note that, for $\epsilon<\epsilon'$, all terms on the left-hand side of \eqref{eq:medium-replace-bd} are contained in $\Lambda_N$.
Using the same argument as in the proof of Lemma \ref{lem:medium-replace}, one can bound it from above by $C_T(\mathcal I_1+\mathcal I_2)$, where
\begin{equation}
\begin{aligned}
\mathcal I_1 &:= \limsup_{\epsilon\downarrow0,N\uparrow\infty} \left\{ \sup_{t\in[0,T]} \sup_{|z|,|z'| \le \epsilon N} \frac1{N^2}\sum_{(x,y) \in B_{N,\epsilon'}} \big|H_{t,x,y}^N - H_{t,x-z,y-z'}^N\big| \right\};\\
\mathcal I_2 &:= \limsup_{\epsilon\downarrow0,N\uparrow\infty} \left\{ \sup_{t\in[0,T]} \sup_{|z|,|z'|  \le \epsilon N} \frac1{N^2}\sum_{(x,y) \in \Delta_{N,\epsilon'}(z,z')} \big|H_{t,x-z,y-z'}^N\big| \right\}.
\end{aligned}
\end{equation}
In the above, $\Delta_{N,\epsilon'}(z,z')$ refers to the symmetric difference of $N B_{\epsilon'}$ and $N B_{\epsilon'}+(z,z')$.
In view of the definition of $H_{t,x,y}^N$ in \eqref{eq:h}, for fixed $\epsilon'>0$,
\begin{equation}
\begin{aligned}
&\sup_{t\in[0,T]} \sup_{|z|,|z'| \le \epsilon N} \sup_{(x,y) \in NB_{\epsilon'}} \big|H_{t,x,y}^N-H_{t,x-z,y-z'}^N\big| \lesssim \epsilon,\\
&\sup_{t\in[0,T]} \sup_{|z|,|z'| \le \epsilon N} \sup_{(x,y) \in \Delta_{N,\epsilon'}} \big|H_{t,x-z,y-z'}\big| \lesssim 1.
\end{aligned}
\end{equation}
Therefore, $\mathcal I_1+\mathcal I_2$ is bounded from above by
\begin{equation}
\limsup_{\epsilon\downarrow0,N\uparrow\infty} \frac{C_{\epsilon'}}{N^2} \left[ \epsilon\cdot\#\big(N B_{\epsilon'}\big) + \sup_{|z|,|z'|<\epsilon N} \#\big(\Delta_{N,\epsilon'}(z,z')\big) \right] = 0
\end{equation}
and the proof is completed.
\end{proof}

\begin{proof}[Proof of Lemma \ref{lem:source}]
In view of \eqref{eq:bd-rate-estimate} and the fact that $\gamma<1$,
\begin{equation}
\limsup_{\epsilon'\downarrow0,N\uparrow\infty} \frac1N \sum_{x=1}^{\epsilon'N-1} \alpha^N(x) \lesssim \limsup_{\epsilon'\downarrow0,N\uparrow\infty} \frac1N \sum_{x=1}^{\epsilon'N} \left( \frac xN \right)^{-\gamma} = 0.
\end{equation}
Together with \eqref{eq:bound-l2-bd}, this yields that
\begin{equation}\label{eq:source-pf-1}
\limsup_{\epsilon'\downarrow0,N\uparrow\infty} \mathbb E_N \left[ \frac1N \int_0^T \sum_{x=1}^{\epsilon'N-1} G_t \left( \frac xN \right) \alpha^N(x) \eta_t(x)\,\d t \right] = 0.
\end{equation}
Notice that similar estimate holds for the summation over $(1-\epsilon')N+1 \le x \le N-1$.
Moreover, recall that \eqref{eq:tail-bd} holds in $L^1((0,1))$, so for any fixed $\epsilon'>0$,
\begin{equation}\label{eq:source-pf-2}
\limsup_{N\uparrow\infty} \mathbb E_N \left[ \frac1N \int_0^T \sum_{x=\epsilon'N}^{(1-\epsilon')N} G_t \left( \frac xN \right) \left| \alpha^N(x) - \alpha \left( \frac xN \right) \right| \eta_t(x)\,\d t \right] = 0.
\end{equation}
Finally, as the function $G_t(u)\alpha(u)$ is uniformly continuous in $[\epsilon',1-\epsilon']$,
\begin{equation}\label{eq:source-pf-3}
\limsup_{\epsilon\downarrow0,N\uparrow\infty} \mathbb E_N \left[ \frac1N \int_0^T \left| \sum_{x=\epsilon'N}^{(1-\epsilon')N} G_t \left( \frac xN \right) \alpha \left( \frac xN \right) \big(\eta_t(x) - \eta_t(x)^{\epsilon N}\big) \right| \,\d t \right] = 0.
\end{equation}
The limit \eqref{eq:source} then follows from the combination of \eqref{eq:source-pf-1}, \eqref{eq:source-pf-2} and \eqref{eq:source-pf-3}.
The other assertion concerning $(\beta^N,\beta)$ is proved similarly.
\end{proof}

\begin{proof}[Proof of Lemma \ref{lem:boundary}]
We prove \eqref{eq:boundary-1}, since \eqref{eq:boundary-2} as well as the criteria corresponding to the right boundary can be shown by the same argument.
Comparing to Lemma \ref{lem:cutting-bd} and \ref{lem:medium-replace-bd}, the absence of smooth test function forces us to apply better estimate on the rate of short jumps.
For small $\epsilon'>\epsilon''>0$, define
\begin{equation}
\Lambda_{\epsilon'} := (0,\epsilon']\times(\epsilon',1), \qquad \Lambda_{\epsilon',\epsilon''} := [\epsilon'',\epsilon']\times[\epsilon'+\epsilon'',1-\epsilon''].
\end{equation}
By \eqref{eq:bound-l2-bd} and the same argument used in Lemma \ref{lem:cutting-bd} and \ref{lem:medium-replace-bd}, it suffices to show that for small but fixed $\epsilon'>0$,
\[
\mathcal J_1 := \limsup_{\epsilon''\downarrow0,N\uparrow\infty} \frac{N^\gamma}N \sum_{(x,y) \in N(\Lambda_{\epsilon'}\setminus\Lambda_{\epsilon',\epsilon''})} \mathfrak{p}(x,y) = 0,
\]
and for small but fixed $\epsilon'>\epsilon''>0$,
\begin{align*}
&\mathcal J_2 := \limsup_{\epsilon\downarrow0,N\uparrow\infty} \left\{ \sup_{|z|,|z'| \le \epsilon N} \frac{N^\gamma}N \sum_{(x,y) \in N\Lambda_{\epsilon',\epsilon''}} \big|\mathfrak{p}(x,y)-\mathfrak{p}(x-z,y-z')\big| \right\} = 0,\\
&\mathcal J_3 := \limsup_{\epsilon\downarrow0,N\uparrow\infty} \left\{ \sup_{|z|,|z'| \le \epsilon N}\frac{N^\gamma}N \sum_{(x,y) \in N\Delta_{\epsilon',\epsilon''}(z,z')} \mathfrak{p}(x-z,y-z') \right\} = 0.
\end{align*}
Here in $\mathcal J_3$, $\Delta_{\epsilon',\epsilon''}(z,z')$ is the symmetric difference of $\Lambda_{\epsilon',\epsilon''}$ and $\Lambda_{\epsilon',\epsilon''}+N^{-1}(z,z')$.

We first claim that $\mathcal J_1$ is negligible.
By \eqref{eq:growth}, since $\gamma<1$,
\begin{equation}
\begin{aligned}
&\limsup_{\epsilon''\downarrow0,N\uparrow\infty} \frac{N^\gamma}N \sum_{x=(\epsilon'-\epsilon'')N}^{\epsilon' N} \sum_{y=\epsilon'N+1}^{x+\epsilon''N} \mathfrak{p}(x,y)\\
\lesssim\, &\limsup_{\epsilon''\downarrow0,N\uparrow\infty} \frac1{N^2} \sum_{x=(\epsilon'-\epsilon'')N}^{\epsilon' N} \sum_{y=\epsilon'N+1}^{x+\epsilon''N} \left( \frac yN-\frac xN \right)^{-1-\gamma}\\
\lesssim\,&\limsup_{\epsilon''\downarrow0} \int_{\epsilon'-\epsilon''}^{\epsilon'} \int_{\epsilon'}^{u+\epsilon''} (v-u)^{-1-\gamma} \,\d v\,\d u = 0.
\end{aligned}
\end{equation}
In view of (A2), what is remained in $\mathcal J_1$ can be written as
\begin{equation}
\limsup_{\epsilon''\downarrow0,N\uparrow\infty} \frac1{N^2} \sum_{(x,y) \in N(\Lambda_{\epsilon'}\setminus\Lambda_{\epsilon',\epsilon''})} \mathfrak p \left( \frac xN,\frac yN \right) \mathds1_{|\frac xN-\frac yN|>\epsilon''}.
\end{equation}
Since $\mathfrak{p}(u,v)\mathds1_{|u-v|>\epsilon''}$ is continuous, it is equal to
\begin{equation}
\lim_{\epsilon''\downarrow0} \iint_{\Lambda_{\epsilon'}\setminus\Lambda_{\epsilon',\epsilon''}} \mathfrak{p}(u,v)\mathds1_{|u-v|>\epsilon''}\,\d v\,\d u,
\end{equation}
and the limit is $0$ due to \eqref{eq:integrability-bd}.

For $\mathcal J_2$, note that $\Lambda_{\epsilon',\epsilon''}$ is bounded and strictly apart from $\{u=v\}$ for any $\epsilon''>0$.
Conditions (A1) and (A2) then yield that for fixed $\epsilon'>\epsilon''>0$,
\begin{equation}
\begin{aligned}
\mathcal J_2 \le \limsup_{\epsilon\downarrow0,N\uparrow\infty} \frac1{N^2} \cdot \#\big(N\Lambda_{\epsilon',\epsilon''}\big) \cdot \epsilon L_{\epsilon',\epsilon''} = 0.
\end{aligned}
\end{equation}

For $\mathcal J_3$, the argument is similar and we get for fixed $\epsilon'>\epsilon''>0$ that
\begin{equation}
\mathcal J_3 \le \limsup_{\epsilon\downarrow0,N\uparrow\infty} \sup_{|z|,|z'| \le \epsilon N} \frac{C''}{N^2} \#\big(N\Delta_{\epsilon',\epsilon''}(z,z')\big) = 0
\end{equation}
The desired conclusion then follows immediately.
\end{proof}

\printbibliography

\end{document}